\documentclass{article}
\usepackage[british]{babel}
\usepackage[utf8]{inputenc} % permette di scrivere lettere accentate nel codice
\usepackage[all]{xy}
\SelectTips{cm}{}

\usepackage{amsmath,amsthm,amssymb,xspace,graphicx,enumitem,xcolor}
\usepackage{geometry}

\edef\cdrestoreat{%%
\noexpand\catcode\lq\noexpand\@=\the\catcode\lq\@}\catcode\lq\@=11
\renewcommand\section{\@startsection {section}{1}{\z@}%
          {-3.25ex\@plus -1ex \@minus -.2ex}%
          {1.5ex \@plus .2ex}%
          {\normalfont\large\bfseries}}
\renewcommand\subsection{\@startsection{subsection}{2}{\z@}%
          {-3.25ex\@plus -1ex \@minus -.2ex}%
          {1.5ex \@plus .2ex}%
          {\normalfont\normalsize\bfseries}}

\mathcode`\:="603A % : = \colon
\mathcode`\<="4268 % < = \langle
\mathcode`\>="5269 % > = \rangle
\mathchardef\colon="303A % re\def for use in :=
\mathchardef\gt="313E  % arithmetic
\mathchardef\lt="313C  % strict order

\let\phi\varphi
\DeclareFontFamily{OT1}{pzc}{}
\DeclareFontShape{OT1}{pzc}{m}{it}{<->s*[1.14]pzcmi7t}{}
\DeclareMathAlphabet{\mathpzc}{OT1}{pzc}{m}{it}

\swapnumbers
\newtheoremstyle{teorema}{\topsep}{\topsep}
{\slshape}{}{\bf}{{\normalfont.}}{.5em}{}
\newtheoremstyle{definizione}{\topsep}{\topsep}
{\normalfont}{}{\bf}{{\normalfont.}}{.5em}{}

\def\nameit#1{#1~}
\def\thx{\nameit{Theorem}}
\def\lmx{\nameit{Lemma}}
\def\prx{\nameit{Proposition}}
\def\crx{\nameit{Corollary}}
\def\dfx{\nameit{Definition}}
\def\ntx{\nameit{Notation}}
\def\rmx{\nameit{Remark}}
\def\rxx{\nameit{Example}}
\def\xmx{\nameit{Example}}
\def\xsx{\nameit{Examples}}

\theoremstyle{teorema}
\newtheorem{thm}{Theorem}[section]
\newtheorem{lem}[thm]{Lemma}
\newtheorem{prp}[thm]{Proposition}
\newtheorem{cor}[thm]{Corollary}
\theoremstyle{definizione}
\newtheorem{dfnt}[thm]{Definition}
\newtheorem{rem}[thm]{Remark}
\newtheorem{nota}[thm]{Notation}

\newtheorem{rxm}[thm]{Examples}

\def\dfn#1{{\bfseries\itshape #1\/}}
\def\ie{{\textit{i.e.}}\xspace}

\def\id#1{\ensuremath{\mathrm{id}_{#1}}}
\def\Id#1{\ensuremath{\mathrm{Id}_{#1}}}

\def\twr{^{\mathbf{2}}}

\def\ct#1{\ensuremath{\mathpzc{#1}}\xspace} % \mathpzc{#1}

\def\CT#1{\ensuremath{{\normalfont\textrm{\bfseries #1}}}\xspace}
\def\Pred#1{{\ensuremath{{\mathpzc{Prd}\kern-.4ex_{{#1}}}}}}

\def\FTR#1{\ensuremath{\mathsf{#1}}}

\def\xynatvr{\ar|>{}|-{{\kern1.2ex\scalebox{.3}{$\bullet$}}}}
\def\iso{\scalebox{.8}{$\sim$}}

\def\RB#1{\mathchoice
  {\rotatebox[origin=c]{180}{$#1$}}
  {\rotatebox[origin=c]{180}{$#1$}}
  {\rotatebox[origin=c]{180}{$\scriptstyle#1$}}
  {\rotatebox[origin=c]{180}{$\scriptscriptstyle#1$}}}
\def\DQ{\RB{E}\kern-.3ex}
\def\DX[#1]#2#3{\ensuremath{\DQ_{#1,#2}^{#3}}}
\def\DL#1#2{\ensuremath{\DQ_{#1}^{#2}}}
\def\D{\@ifnextchar[{\DX}{\DL}}
\def\BQ{\RB{A}\kern-.6ex}
\def\BX[#1]#2#3{\ensuremath{\BQ_{#1,#2}^{#3}}}
\def\BL#1#2{\ensuremath{\BQ_{#1}^{#2}}}
\def\B{\@ifnextchar[{\BX}{\BL}}
\cdrestoreat

\def\Gr#1{\ensuremath{\mathop{{}\ct{G}\kern-.5ex}({#1})}}

\def\cmp#1{\ensuremath{\{\kern-2.5pt|{#1}|\kern-2.5pt\}}}
\def\Set{\ct{Set}}
\def\Cat{\ct{Cat}}

\let\land\wedge
\def\Implies{\ensuremath{\Rightarrow}}
\def\Iff{\ensuremath{\Leftrightarrow}}
\let\tt\top
\def\Gpd{\ct{Gpd}}
\def\domF{\ensuremath{\FTR{Pr}_1}\xspace}
\def\dom{\FTR{dom}\xspace}
\def\cod{\FTR{cod}\xspace}
\def\idd#1{\ensuremath{\FTR{id}_{#1}}}

\def\1{\textbf{1}}
\def\Fam{\CT{Fam}}
\def\FamS{\Fam(\Set)}
\def\FamC{\Fam(\ct{C})}
\def\fam#1#2#3{\ensuremath{(#1,(#3_{#2})_{#2 \in #1})}\xspace}
\def\famI{\fam{I}{i}{A}}
\def\famJ{\fam{I'}{j}{A'}}

\def\larr{\mathrel{\xymatrix@1@=3.5ex{*=0{}\ar[];[r]&*=0{}}}}
\def\to{\mathrel{\xymatrix@1@=2.5ex{*=0{}\ar[];[r]&*=0{}}}}
\def\tnat{\stackrel{\smash{\raisebox{-.2ex}{.\kern.5ex}}}\to}
\def\isoa{\stackrel{\smash{\raisebox{-.2ex}{\iso\kern.5ex}}}\to}
\def\xyisoa{\stackrel{\smash{\raisebox{-.2ex}{\iso\kern.5ex}}}
 \longrightarrow}

\let\ftr\larr

\def\commentmark#1{\makebox[0pt][l]
{\color{#1}\kern-.3ex\textbullet}}

\def\Map#1{\ensuremath{\ct{EF}\kern-.4ex_{#1}}}
\def\DMap#1{\ensuremath{\ct{M}\kern-.4ex_{#1}}}

\def\xycar{\ar@{-)}}
\def\car{\mathrel{\xymatrix@C=2.5ex{ *=0{} \xycar[r] & *=0{} }}}
\def\xycocar{\ar@{-+}}
\def\cocar{\mathrel{\xymatrix@C=2ex{ *=0{} \xycocar[r] & *=0{} }}}

\def\ridx#1#2{#1{}^*#2} % reindexed #2 along #1
\def\pr#1{\ensuremath{{\mathrm{pr}}_{#1}}} % product projection in the codomain of a fibration with products
\def\carl#1{^{\Rsh\!#1}} % cartesian lift into #1
\def\ladjD#1{\ensuremath{\mathchoice
 {\rotatebox[origin=c]{180}{E}_{#1}}
 {\rotatebox[origin=c]{180}{E}_{#1}}
 {\rotatebox[origin=c]{180}{\scriptsize E}_{\scriptscriptstyle #1}}
 {\relax}}}% left adjoint to reindexing along #1
\def\cocaDp#1#2{\ensuremath{\cocaD{#2}^{#1}}} % equality arrow of #2 with parameter #1
\def\cocaD#1{\ensuremath{\delta_{#1}}} % equality arrow of #1
\def\cocD#1{\ensuremath{\partial_{#1}}} % cocartesian arrow over \pr{1,1}
\def\eqobj#1{\ensuremath{\mathrm{I}_{#1}}} % equality object of #1
\def\desdg#1#2#3{(\ridx{#3} #2) \land \eqobj{#1}} % general syntax for descent domain of #2 over #1
\def\desd#1#2{\desdg{#1}{#2}{\pr 1}} % descent domain of #2 over #1
\def\desdgp#1#2#3#4{\ensuremath{(\ridx{#4} #2) \land (\ridx{#3}{\eqobj{#1}})}} % general syntax for parametrised descent domain of #2 over #1
\def\desdp#1#2{\desdgp{#1}{#2}{\pr{2,3}}{\pr{1,2}}} % parametrised descent domain of #2 over #1
\DeclareMathOperator{\glprd}{\resizebox{1.4ex}{!}{$\boxtimes$}} % product in the domain of a fibration with products
\def\glpair<#1,#2>{<\!| #1, #2 |\!>}
\def\xeqobj#1#2{\ensuremath{\eqobj #1 \glprd \eqobj #2}} % product of equalities
\def\vrt#1{\ensuremath{\widehat{#1}}} % vertical arrow induced by #1
\def\trmf#1{\ensuremath{\top_{\!\!#1}}} % terminal object in the fibre over #1
\def\trspp#1#2{\ensuremath{\mathrm{t}^{#1}_{#2}}} % transport of #2 with parameter #1
\def\trsp#1{\trspp{}{#1}} % transport of #1
\def\gdiagcl{\ensuremath{\Delta}\xspace}
\def\BCFRcl#1#2{\ensuremath{\Xi_{#1}^{#2}}}
\def\cocaDpcl#1{\ensuremath{\BCFRcl{\pr{1,1}}{\pr 1}(\cocD{#1})}}
\def\cocaCl{\ensuremath{\Lambda}\xspace}

\def\iff{if and only if\xspace}

\def\testa{}
\def\fibepi#1{\relax\def\testx{#1}\relax
locally epi\ifx\testa\testx\relax\else{} with respect to \ensuremath{#1}\fi\relax}
\def\fibepic#1{\relax\def\testx{#1}\relax
locally epic\ifx\testa\testx\relax\else{} with respect to \ensuremath{#1}\fi\relax}

\def\FRar{pairable\xspace}
\def\BCarL{product-stable\xspace}
\def\BCar{\BCarL}
\def\FRop{split pairing\xspace}
\def\BCop{parametrised reindexing\xspace}

\def\arrcl#1{\ct{#1}}
\def\fibrcl{\arrcl{F}}
\def\weqcl{\arrcl{W}}
\def\leftcl{\arrcl{L}}
\def\rightcl{\arrcl{R}}
\def\rst#1{\mathbin{\restriction}_{#1}}
\def\transporter{transporter\xspace}

\def\ple#1{\ensuremath{\left\langle #1 \right\rangle}}
\def\ple#1{\ensuremath{\langle #1 \rangle}}
\def\refl#1{\ensuremath{\mathrm{r}_{#1}}}

\def\awfs{algebraic weak factorisation system\xspace}

\def\awfsR{\FTR{R}\xspace}
\def\awfsL{\FTR{L}\xspace}
\def\awfsM{\FTR{M}\xspace}
\def\awfsR{\ensuremath{R}\xspace}
\def\awfsL{\ensuremath{L}\xspace}
\def\awfsM{\ensuremath{M}\xspace}
\def\awfsRc{\FTR{R}\xspace}
\def\awfsLc{\FTR{L}\xspace}
\def\awfsMc{\FTR{M}\xspace}
\def\awfsRg{\FTR{R}\xspace}
\def\awfsLg{\FTR{L}\xspace}
\def\awfsMg{\FTR{M}\xspace}
\def\RcAlg{\ensuremath{\awfsRc\text{-}\CT{Alg}}\xspace}
\def\RcMap{\ensuremath{\awfsRc\text{-}\CT{Map}}\xspace}
\def\RgAlg{\ensuremath{\awfsRg\text{-}\CT{Alg}}\xspace}
\def\RgMap{\ensuremath{\awfsRg\text{-}\CT{Map}}\xspace}
\def\RAlg{\ensuremath{\awfsR\text{-}\CT{Alg}}\xspace}
\def\RMap{\ensuremath{\awfsR\text{-}\CT{Map}}\xspace}
\def\LcAlg{\ensuremath{\awfsL\text{-}\CT{Coalg}}\xspace}
\def\LMap{\ensuremath{\awfsL\text{-}\CT{Map}}\xspace}

\def\sid#1{\Id{\sct{#1}}}
\def\sdom#1{\ensuremath{\FTR{d}_{\sct{#1}}}}
\def\scod#1{\ensuremath{\FTR{c}_{\sct{#1}}}}
\def\srefl#1{\ensuremath{\mathsf{r}}_{\sct{#1}}}
\def\sstr#1{\ensuremath{\mathsf{s}_{\sct{#1}}}}

\def\isosct#1{\ensuremath{\mathrm{Iso}(\sct{#1})}}
\def\sct#1{\ensuremath{\mathbb{#1}}}

\def\blank{\mathrm{-}}

\newenvironment{innerenu}[1]%
{\begin{enumerate}
 [label={\normalfont\alph*.},ref=#1\normalfont.\alph*]}%
{\end{enumerate}}

\newcommand{\xycentre}[2][]{\vcenter{\hbox{\xymatrix#1{#2}}}}

\date{}

\begin{document}

\title{A characterisation of elementary fibrations}
\author{Jacopo Emmenegger
\thanks{School of Computer Science,
University of Birmingham,
Birmingham B15 2TT, UK,\hfill\mbox{}
email:~\texttt{op.emmen@gmail.com}}
\and
Fabio Pasquali
\thanks{DIMA, Universit\`a degli Studi di Genova,
via Dodecaneso 35, 16146 Genova, Italy,\hfill\mbox{}
email:~\texttt{pasquali@dima.unige.it}}
\and
Giuseppe Rosolini
\thanks{DIMA, Universit\`a degli Studi di Genova,
via Dodecaneso 35, 16146 Genova, Italy,\hfill\mbox{}
email:~\texttt{rosolini@unige.it}}
}

\maketitle

\begin{abstract}
Grothendieck fibrations provide a unifying algebraic framework that underlies
the treatment of various form of logics, such as 
first order logic, higher order logics and dependent type theories.
In the categorical approach to logic proposed by Lawvere,
which systematically uses adjoints to describe the logical operations,
equality is presented in the form of a left adjoint to reindexing along a diagonal arrows in the base.
Taking advantage of the modular perspective provided by category theory,
one can look at those Grothendieck fibrations which sustain just the structure of equality,
the so-called elementary fibrations, aka fibrations with equality.

The present paper provides a characterisation of elementary fibrations
based on particular structures in the fibres,
called transporters.
The characterisation is a substantial generalisation of the one already available for faithful fibrations.
There is a close resemblance between transporters
and the structures used in the semantics
of the identity type of Martin-L\"of type theory.
We close the paper by comparing the two.
\end{abstract}

\section{Introduction}

Grothendieck fibrations provide a unifying algebraic framework that underlies
the treatment of various form of logics, such as 
first order logic, higher order logics and dependent type theories. The approach dates back to the seminal work of Bill Lawvere on functorial semantics, in
particular his work on hyperdoctrines~\cite{LawvereF:adjif,LawvereF:equhcs}.
The structure consists of a functor $P:\ct{A}\ftr\ct{T}$ which is a fibration (we recall the
definition in Section~\ref{two});
the base \ct{T} is to be understood as the universe of discourse given by the sorts and the
terms of the theory, while, for an object $X$ in \ct{T},
the category $\ct{A}_X$ of objects and arrows over $X$
presents the properties,
the ``attributes'' in the words of Lawvere, of the object $X$ and the relevant
entailments between them. Logic, in a sense, appears via properties of the fibration which
involve adjoints to basic functors such as
``doubling the objects''
$$\xymatrix@C=3em@R=3em{\ct{A}\ar[rr]^-{\ple{\Id{\ct{A}},\Id{\ct{A}}}}\ar[rd]_-{P}&&
\ct{A}\times_{\ct{T}}\ct{A}\ar[dl]^-{P\times_{\ct{T}}P}\\
&\ct{T}}$$
see also \cite{JacobsB:catltt}.

Some time after Lawvere proffers his structural view of logic, Per Martin-L\"of puts forward
the proposal of a dependent type theory where the constructions on types match very closely
the construction on formulas, see~\cite{MartinLofP:inttot,MartinLofP:conmcp,ML84}.
Expressly, he refers to
preliminary work of Dana Scott~\cite{ScottD:conv} on a tentative calculus of proofs as
constructions; in turn Scott acknowledges the earlier attempts of Lawvere as they all address
the idea to give propositions the same status as types (this idea indeed would later become
known with similar, more fashionable locutions).

One of the many remarkable features of the Lawverian proposal for a categorical approach to
logic is the systematic use of adjoint functors to describe the logical operations, in
particular the realisation that
equality comes in the form of a left adjoint to certain,
structural, reindexing functors,
see~\cite{LawvereF:equhcs}.
The counterpart in type theory appears in the literature
some fifteen years later in the form of ``propositional equality'',
see~\cite{MartinLofP:conmcp}, but also~\cite{PMTT} which calls it ``intensional equality'',
often it is simply called ``identity type''.
Roughly speaking, given a type $T$, the identity type for $T$
can be understood as a family of ``proofs of equality''
between terms of type $T$,
inductively generated by reflexivity with a fixed endpoint.

Quite remarkably, some fundamental work
in the categorical semantics of type theory resurrects intuitions about equality
which featured prominently in~\cite{LawvereF:equhcs}:
the groupoid interpretation of Martin-L\"of Type Theory
of~\cite{HofmannStreicher1998} and
the interpretation of identity types as homotopies of~\cite{AwodeyWarren2009}.
This cannot be a surprise. The Lawverian approach intended to conjoin the geometric view and
the abstract view of logic, in particular equality had to be as powerful as to admit that two
homotopical paths could be considered, in fact argued as if they were, equal, as powerful as
to admit that isomorphic structures could be considered, in fact transformed into, the same.

Taking advantage of the modular perspective provided by category theory, one can look at
those Grothendieck fibrations which sustain just the structure of equality. These are called
``fibrations with equality satisfying Frobenius'' in \cite[3.4.1]{JacobsB:catltt}. In some
previous work \cite{EmmeneggerJ:eledac} we concentrated on the particular case of
\emph{faithful} fibrations with equality satisfying Frobenius---which we renamed
elementary doctrines referring to Lawvere's original terminology---%
and we shall follow
suit and christen \dfn{elementary fibration} the notion defined in
\cite[3.4.1]{JacobsB:catltt}.

Motivated by the characterisation obtained in \cite{EmmeneggerJ:eledac},
we present a characterisation of elementary fibrations
that contributes to shed light on the relationship
between the approaches to equality via category theory and via type theory.
The characterisation is based on a structure
which consists of a family of internal actions on the ``attributes''.
Slightly more precisely the family consists,
on each of the objects in the total category \ct{A}
of the fibration $P:\ct{A}\ftr\ct{T}$,
of an algebra map for a certain pointed endofunctor on \ct{A}
over the endofunctor on \ct{T} which maps
an object $X$ to the product $X \times X$.
However, we find it convenient
to introduce such a structure in a more elementary way,
by a gradual strengthening of weaker structures.
This simplifies the comparison
with the type-theoretic approach:
as it will become clear,
in type-theoretic terms this family of actions can be understood
as a transport along a proof of equality.

The complete statement of our main result lists other equivalent
characterisations of an elementary fibration
and the proof builds on the well-known observation that
existence of left adjoints to reindexing is equivalent
to existence of cocartesian lifts.
In the case of faithful fibrations, the characterisation reduces to
the well-known characterisation of first-order equality
as a reflexive and substitutive relation stable under products.
The parallel with the results in~\cite{EmmeneggerJ:eledac}
also suggests a coalgebraic characterisation of elementary fibrations
and this will appear in a subsequent paper in preparation.

We also apply the characterisation to discuss the relationship
between elementary fibrations and fibrations coming from
the homotopical semantics of identity types.
Since the work by Awodey and Warren~\cite{AwodeyWarren2009}
and Gambino and Garner~\cite{GambinoGarner2008},
weak factorisation systems,
best known as part of Quillen model categories~\cite{Quillen},
have proved to provide a suitable framework
to account for the inductive nature of identity types,
see for instance~\cite{KapulkinLumsdaine2012,Shulman2015,Joyal2017,vdBerg2018a}.
Recall that
a \dfn{weak factorisation system} on a category \ct{C}
consists of two classes of arrows \leftcl and \rightcl
which contain all the isomorphisms and are closed under retracts,
and such that each arrow in \ct{C} factors as an arrow in \leftcl
followed by an arrow in \rightcl,
and for each square
\[\xymatrix@=1.5em{
A	\ar[d]_-l \ar[r]	&	X	\ar[d]^-r
\\
B	\ar[r]	&	Y
}\]
with $l \in \leftcl$ and $r \in \rightcl$
there is a diagonal filler $d: B \to X$
making the two triangles commute.
Note that, contrary to orthogonal factorisation systems,
such a filler is not required to be unique.
A weak factorisation system determines a fibration:
the full subfibration of the codomain fibration
on those arrows in the right class \rightcl.
When the factorisation system is orthogonal (plus some additional conditions)
that fibration is always elementary.
But, not surprisingly,
for weak factorisation system
it is rarely so.

A weak factorisation system seems to lack the structure
to isolate a suitable non-full subcategory structure on \rightcl.
It also seems to lack the structure 
to soundly interpret the rules of the identity type in its associated fibration.
Indeed, what in a weak factorisation system is a property of an arrow,
namely the existence of diagonal fillers,
in type theory is structure on a type,
namely a choice of a family of terms.
This has made it impossible so far to interpret,
for example, (one of) the substitution rules for the identity type.
As shown by van den Berg and Garner~\cite{vdBergGarner2012},
this gap can be overcome by imposing some algebraic conditions
on a choice of fillers for just certain squares.
Further work by Garner and 
others~\cite{Garner2009,GambinoSattler2017,GambinoLarrea2019}
identified a suitable framework to express these conditions
in algebraic weak factorisation systems~\cite{GrandisTholen},
see also~\cite{BourkeGarner}.

The richer structure
of algebraic weak factorisation systems,
whose definition we recall in Section~\ref{Applications},
produces also more structured fibrations.
In particular, this is the case with
the algebraic weak factorisation system on the category of small categories \Cat
(and its full subcategory \Gpd on the groupoids)
whose underlying weak factorisation system is
the one of acyclic cofibrations and fibrations from the canonical,
or ``folk'', model structure on \Cat (and \Gpd).
We prove that the fibration of algebras
associated to the algebraic weak factorisation system
is elementary. 
In the case of the algebraic weak factorisation system on \Gpd,
the associated full comprehension category is the Hofmann--Streicher
groupoid model from~\cite{HofmannStreicher1998},
see~\cite{GambinoLarrea2019} for the relation to the groupoid model,
and~\cite{EmmeneggerJ:elefog} for a discussion of the enriched case.

There are two side results worth noticing:
the first one is that we find models of dependent
type theory where equality is given as an adjunction,
but the identity types are not trivial.
The second one is that the intensional model of type theory given by
the groupoid interpretation can be reconnected
to one of the original suggestions of Lawvere in \cite{LawvereF:equhcs}.

In Section~\ref{two} we recall notations and results from the theory
of fibrations which are necessary for the following sections. In
Section~\ref{three} we introduce the structure of transporters in a
fibration and prove some elementary properties of these. These are put
to use in Section~\ref{four} which contains the main characterisation
theorem. Finally Section~\ref{Applications} contains applications to
algebraic weak factorisation systems.
In particular, we illustrate the connection between the
groupoid hyperdoctrine of Lawvere and
the groupoid model of Hofmann and Streicher.
%in particular we show the connection between the
%groupoid hyperdoctrine of Lawvere and the intentional interpretation of type theory of
%Hofmann and Streicher.

\section{Preliminaries}\label{two}

Let $p:\ct{E}\ftr\ct{B} $ be a functor.
An arrow $\varphi$ in \ct{E} is said
to be \dfn{over} an arrow $f$ in \ct{B} when $p(\varphi) = f$.
For $X$ in \ct{B}, the fibre $\ct{E}_X$
is the subcategory of $\ct{E}$ of arrows over $\id{X}$.
In particular, an object $A$ in \ct{E} is said to be
\dfn{over} $X$ when $p(A) = X$.

Recall that an arrow $ \varphi : A \to B $ is \dfn{cartesian} if,
for every $ \chi : A' \to B $ such that $ p(\chi) $ factors through $ p(\varphi) $ via an arrow $ g : X' \to X $,
there is a unique $ \psi : A' \to A $ over $ g $ such that $ \varphi \psi = \chi $,
as in the left-hand diagram below.
And an arrow $ \theta : A \to B $ is \dfn{cocartesian} if it satisfies the dual universal property of cartesian arrows
depicted in the right-hand diagram below.
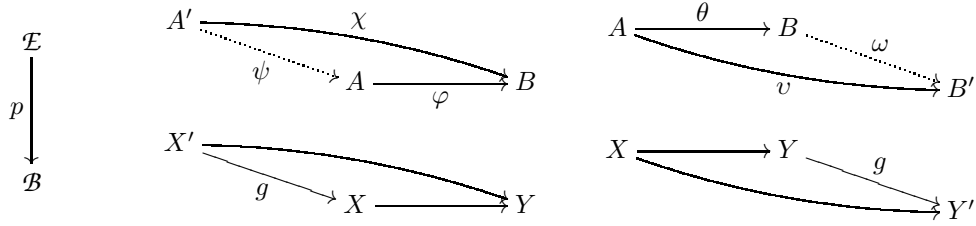
\[
\xycentre[@R=4em]{
\ct{E}	\ar[d]_-p	\\	\ct{B}	}
\qquad\qquad
\xycentre[@R=1em@C=5em]{
A'\ar@{.>}[dr]_{\psi} \ar@/^/[drr]^{\chi}	&&\\
&	A	\ar[r]_{\varphi}	&	B	\\
X'	\ar[dr]_g \ar@/^/[drr]	&&\\
&	X	\ar[r]	&	Y	}
\qquad
\xycentre[@R=1em@C=5em]{
A	\ar@/_/[drr]_{\upsilon} \ar[r]^{\theta}	&	B	\ar@{.>}[dr]^{\omega}	&\\
&&	B'	\\
X	\ar@/_/[drr] \ar[r]	&	Y	\ar[dr]^g	&\\
&&	Y'	}
\]

Once we fix an arrow $ f : X \to Y $ in \ct{B} and an object $B$ in $ \ct{E}_Y $,
a cartesian arrow $\varphi : A \to B$ over $f$ is uniquely determined up to isomorphism,
\ie if $\phi' : A' \to B$ is cartesian over $f$,
then there is a unique iso $\psi : A' \to A$ such that $\varphi \psi = \varphi'$.

Clearly, every property of cartesian arrows applies dually to cocartesian arrows.
So for an arrow $ f : X \to Y $ in \ct{B} and an object $A$ in $ \ct{E}_X $,
a cocartesian arrow $\theta : A \to B$ over $f$ is uniquely determined up to isomorphism.

In the following, we write cartesian arrows as $\car$, and
cocartesian arrows as $\cocar{}$.

A functor $p:\ct{E}\ftr\ct{B}$ is a \dfn{fibration} if,
for every arrow $f:X\to Y$ in \ct{B} and for every object $A$ in
$\ct{E}_Y$, there is a \dfn{cartesian lift} of $f$ into $A$, that is,
an object $\ridx{f}{\!A}$ and a cartesian arrow
$f\carl{A}:\ridx{f}{\!A}\car A$ over $f$.
A \dfn{cleavage} for the fibration $p$ is a choice of a cartesian lift
for each arrow $f:X\to Y$ in \ct{B} and object $B$ in $\ct{E}_Y$,
and a \dfn{cloven fibration} is a fibration equipped with a cleavage.
In a cloven fibration, for every $f:X\to Y$ in \ct{B},
there is a functor $\ridx{f}{}:\ct{E}_Y\to\ct{E}_X$
called \dfn{reindexing} along $f$.
Henceforth we assume that fibrations can be endowed with a
cleavage.

\begin{rem}\label{brem}
It is well-known that, for the fibration $p:\ct{E}\ftr\ct{B}$, an arrow
$f:X\to Y$ in \ct{B} has cocartesian lifts if and only if the
reindexing functor $\ridx{f}{}:\ct{E}_Y\ftr\ct{E}_X$ has a left
adjoint. The value of the left adjoint at an object $A$ over $X$ can
be chosen as the codomain $A'$ of a cocartesian lift $A\cocar A'$
of $f:X\to Y$ at $A$. Conversely, the cocartesian lift is given by
the composition 
\[
\xymatrix@C=4em{
A	\ar[r]^-{\eta_A}
&	\ridx{f}{(\ladjD{f}(A))}	\xycar[r]^-{f\carl{\ladjD{f}(A)}}
&	\ladjD{f}(A)	}
\]
of the unit $\eta_A:A\to\ridx{f}{(\ladjD{f}(A))}$ of the
adjunction $\ladjD f\dashv\ridx{f}{}$ and the cartesian lift of
$f$.
\end{rem}

Fibrations are ubiquitous in mathematics and the list of examples is
endless. Since our aim is to characterise those fibrations which
encode a proof-relevant notion of equality, we choose the following
classes of examples.

\begin{rxm}\label{runex-fib}
(a)
Given a category \ct{C}, let
 $\FamC$ be the category whose objects are set-indexed
families of objects in \ct{C},
\ie pairs \famI
where $I$ is a set and $A_i$ is an object in \ct{C}, for $i \in I$,
and an arrow from \famI to \famJ
is a pair $(f,\phi)$
where $f:I \to I'$ is a function
and $\phi = (\phi_i : A_i \to A_{f(i)})_{i\in I}$ is a family of
arrows in \ct{C}, see \cite[1.2.1]{JacobsB:catltt}.

Equivalently, an object of $\FamC$ is a functor
$A:I\to \ct{C}$ where $I$ is a set seen as a discrete category, and an
arrow from $A$ to $B:J\to \ct{C}$ is a pair $(f,\phi)$ where $f:I\to
J$ is a function and $\phi:A \tnat Bf$ is a natural
transformation as in the diagram 
\[
\xymatrix@C=4em@R=1em{
I\ar[dd]_f\ar[rd]^(.4){A}_(.4){}="P"&&\\
& \ct{C}.&\\
J \ar[ru]_(.4){B}^(.4){}="R"&\xynatvr_-{\phi}"P";"R"&}
\]
Since $I$ is discrete, all commutative diagrams for naturality are trivial.

The functor $\domF:\FamC\ftr\Set$ that sends
$(f,\phi) : A \to B$
to $f:I\to J$ is a  fibration.
An arrow $(f,\id{}):Bf\to B$ is cartesian into $B$ over $f:I\to J$.
Note that $\Fam(\ct{1})\equiv\Set$ and that the fibration
$\domF:\FamC\ftr\Set$ is isomorphic to
$\Fam(!):\Fam(\ct{C})\ftr\Fam(\ct{1})$,
where $!: \ct{C}\to\ct{1}$ is the unique functor.
\smallskip

\noindent(b)
Let $\arrcl{F}$ be a full subcategory of $\ct{C}\twr$,
so that an arrow $f : a \to b$ in \arrcl{F},
where $a$ and $b$ are arrows in \ct{C},
is a commutative square
\[\xymatrix{
A	\ar[r]^-{f_1} \ar[d]_-a	&	B	\ar[d]^-b
\\
X	\ar[r]^-{f_0}	&	Y.
}\]
Assume that,
for every $f : X \to Y$ in \ct{C} and  $g : B \to Y$ in $\arrcl{F}$,
there is a pullback square
\[
\xymatrix{
P	\ar[d]_-{g'} \ar[r]	&	B	\ar[d]^-g	\\
X	\ar[r]_-f	&	Y	}
\]
and $g' \in \arrcl{F}$. In that case
the composite
\[
\xymatrix{
\arrcl{F}	\ar@{^{(}->}[r] \ar@/_1em/[rr]_-{\cod\rst{\arrcl{F}}}
&	\ct{C}\twr	\ar[r]^-{\cod}	&	\ct{C}
}
\]
is a fibration.
Given $f : X \to Y$ in \ct{C}, a cartesian lift into the object $g : B \to Y$,
where $g \in \arrcl{F}$, is a pullback square as the one above.
Since \arrcl{F} is full, in the following we shall often confuse it with its collection of objects,
\ie a collection of arrows of \ct{C}.

When \ct{C} has pullbacks we can choose \arrcl{F} even as $\ct{C}\twr$ itself.
In the particular case of $\ct{C} \colon= \Set$ the example in (a)
come to be the same as the example in (b) since there is an
equivalence
\[\xymatrix@C=1.5em{\FamS\ar[dr]_-{\domF}\ar@{3-3-}[rr]&
&\Set\twr\ar[dl]^-{\cod}\\
&\Set&}\]
\end{rxm}

Recall that a fibration $ p : \ct{E} \ftr \ct{B} $ 
\dfn{has finite products} if the base \ct{B} has finite products as
well as each fibre $ \ct{E}_X $, and each reindexing functor preserves
products. Equivalently, both \ct{B} and \ct{E} have finite products
and $p$ preserves them. 

\begin{nota}
We do not require a functorial denotation for products;
when we write $1$ we refer to any terminal object in \ct{B} and,
similarly for objects $X$ and $Y$ in \ct{B},
when we write $ X \times Y $,
$ \pr 1 : X \times Y \to X $ and $\pr 2 : X \times Y \to Y$,
we refer to any diagram of binary products in \ct{B}.
Universal arrows into a product
induced by lists of arrows shall be denoted as $ < f_1 ,\dots, f_n > $,
but lists of projections
$ < \pr{i_1}, \dots, \pr{i_k} > $
will always be abbreviated as $ \pr{i_1,\dots,i_k} $.
In particular, as an object $ X $ is a product of length 1,
sometimes we find it convenient to denote the identity on $ X $ as $ \pr 1 $,
the diagonal $ X \to X \times X $ as $ \pr{1,1} $
and the unique $X \to 1$ as \pr{0}.
As the notation is ambiguous,
we shall always indicate domain and codomain of lists of projections
and sometimes we may distinguish projections
decorating some of them with a prime symbol.

We shall employ a similar notation for
binary products and projections in a fibre $ \ct{E}_X $,
as \trmf{X}, $A\land_XB$, $\pi_1:A\land_XB\to A$ and
$\pi_2:A\land_XB\to B$,
dropping the subscript in \trmf{X} and $\land_X$ when it is
clear from the context.
Moreover, given objects $A$ in $\ct{E}_X$ and $B$ in $\ct{E}_Y$ write
\[
A \glprd B  \colon= (\ridx{\pr 1} A) \land_{X \times Y} (\ridx{\pr 2} B)
\]
for the product of $A$ and $B$ in the total category \ct{E}.
Given a third object $ C $ and two arrows $\varphi_1:C\to A$ and
$\varphi_2:C\to B$, we  denote the induced arrow into $A\glprd B$
also as $ <\varphi_1,\varphi_2>$. 
\end{nota}

\begin{rxm}\label{runex-prim}
(a)
Consider the fibration $\domF: \FamC \ftr \Set$ defined in
\rxx\ref{runex-fib}(a) and 
suppose that \ct{C} has finite products.
Then the fibration $\domF$ has finite products.
Indeed a product of the two families $A:I\to \ct{C}$ and
$B:I\to\ct{C}$ in the fibre $\FamC_I$ is the family 
$A\land B:I\to\ct{C}$ where $(A\land B)_i$ is $A_i\times B_i$
with projections 
$(\id I,\pr1)$ and $(\id I,\pr2)$ where $(\pr1)_i$ is the first
projection $A_i\times B_i\to A_i$. 
A terminal object in $\FamC_I$ is given by the family $1:I\to \ct{C}$
which is constantly a chosen terminal object of \ct{C}.
\smallskip

\noindent (b)
Assume that the base \ct{C} of the fibration $\cod\rst{\arrcl{F}}$
defined in \rxx\ref{runex-fib}(b) has finite products. Then the
fibration $\cod\rst{\arrcl{F}}$
has finite products, and in the fibres these are given by pullback of
arrows in \arrcl{F}.
\end{rxm}

\section{Transporters}\label{three}

This section presents
a structure that will be useful in the characterisation in \thx\ref{mainthm},
providing along the way examples and some instrumental results.

\begin{dfnt}
Let $p : \ct{E} \ftr \ct{B}$ be a fibration with finite products
and consider an object $X$ in \ct{B}.
A \dfn{loop} on $X$ consists of an object $\eqobj X$ over $X\times X$ and
an arrow $\cocD X:\trmf X\to\eqobj X$ over $\pr{1,1}:X\to X\times X$.
\end{dfnt}

\begin{dfnt}
A fibration $p : \ct{E} \ftr \ct{B}$ with finite products
\dfn{has loops} if it has a loop on every $X$.
The fibration $p$ has \dfn{productive loops} if
\begin{enumerate}[label=(\roman*)]
\item
there is a loop $\tt_X \overset{\cocD{X}}{\longrightarrow} \eqobj{X}$
on every $X$ in \ct{B};
\item
for every $X$ and $Y$ in \ct{B}, there is a vertical arrow
$\xymatrix@1{\chi_{X,Y}:\xeqobj X Y\ar[r]&\eqobj{X\times Y}}$.
\end{enumerate}
It has \dfn{strictly productive loops} if moreover
the following diagram commutes for every $X$ and $Y$.
\[\xycentre[@C=4em]{
&\trmf{X\times Y}
\ar[dl]_{\cocD X \glprd \cocD Y} \ar[dr]^{\cocD{X \times Y}}
&\\
\xeqobj X Y	\ar[rr]^-{\chi_{X,Y}}
&&	\eqobj{X\times Y}
}\]
\end{dfnt}

\begin{nota}\label{nota-cocar}
Let $p:\ct{E}\ftr\ct{B}$ be a fibration with loops.  Given $A$ over $X$, we find it convenient
to write $\cocaD A$ for the arrow
$<\pr{1,1}\carl{(\ridx{\pr1}A)},\cocD X!_A> : A \to \desd X A$. We shall also need a
parametric version of it, as for instance in \dfx\ref{def-partrsp}.
When $A$ is an object over $Y\times X$ we write $\cocaDp YA$ for the arrow
$<\pr{1,2,2}\carl{(\ridx{\pr{1,2}} A)},\theta> : A \to \desdp X A$,
where $ \theta : A \to \ridx{\pr{2,3}}{\eqobj X} $ is the unique arrow
over \pr{1,2,2} obtained by cartesianness of
$\ridx{\pr{2,3}}{\eqobj X}\car\eqobj X$
as shown in the diagram
\[\xymatrix@C=3em@R=1ex{A\ar[dd]^-{!_A}\ar@{-->}[rrd]^{\theta}&
  \\
 &&\ridx{\pr{2,3}}{\eqobj X}\xycar[rdd]\\
\trmf{Y\times X}\ar[rd]^-{\pr2\carl{\trmf X}}&\\
&\trmf X\ar[rr]^-{\cocD X}&&\eqobj X\\ 
Y\times X\ar[rd]_-{\pr2}\ar[rr]_(.6){\pr{1,2,2}}&&
Y\times X\times X\ar[rd]^-{\pr{2,3}}\\
&X\ar[rr]_-{\pr{1,1}}&&X\times X.}\]

We write \gdiagcl for the class of arrows of the form
$\pr{1,2,2}:Y\times X\to Y\times X\times X$, for $Y,X$ in \ct{B},
and \cocaCl for the class consisting of all arrows of the form
\cocaDp{Y}{A} in \ct{E} defined above,
for $A$ over $Y \times X$ and $Y,X$ in \ct{B}.
\end{nota}

\begin{dfnt}
Let $p : \ct{E} \ftr \ct{B}$ be a fibration with finite products,
$X$ an object in \ct{B} and $\cocD{X}:\tt_X \to \eqobj{X}$ a loop on $X$.
Given
$A$ over $X$,
a \dfn{carrier for the loop \cocD{X} at $A$} is an arrow
$\trsp A : \desd X A \to A$
over $\pr 2:X\times X\to X$.
The carrier $\trsp A$ is \dfn{strict} if
$\trsp{A}\cocaD{A} = \id{A}$.
\end{dfnt}

\begin{dfnt}\label{E4prd}
Let $p : \ct{E} \ftr \ct{B}$ be a fibration with finite products
and $X$ an object in \ct{B}.
A \dfn{\transporter on $X$} consists of
\begin{enumerate}[label=(\roman*)]
\item
a loop $\cocD X:\trmf X\to\eqobj X$ on $X$;
\item
for every $A$ over $X$, a carrier for \cocD{X}.
\end{enumerate}
A \transporter is \dfn{strict} if every carrier is strict.

A fibration $p : \ct{E} \to \ct{B}$
\dfn{has $($strict$)$ \transporter{}s}
if it has a (strict) \transporter on each $X$ in \ct{B}.

A fibration $p : \ct{E} \to \ct{B}$
\dfn{has $($strict$)$ productive \transporter{}s} if
\begin{enumerate}[label=(\roman*)]
\item
it has (strict) transporters;
\item
loops are (strictly) productive.
\end{enumerate}
\end{dfnt}

\begin{rem}\label{mah}
Strict productive \transporter{}s give a commutative diagram
\[\xymatrix@C=4em{&\trmf{X\times Y}
\ar[dl]_{\cocD X \glprd \cocD Y}\ar[dr]^{\cocD{X \times Y}}&\\
\xeqobj X Y\ar[rr];[]_-{\omega_{X,Y}}&&\eqobj{X\times Y}}\]
for a canonical arrow
$\omega_{X,Y} : \eqobj{X \times Y} \to \eqobj{X} \glprd \eqobj{Y}$.
The reader will see this in the proof of \ref{E4}\Implies\ref{E5} in
\thx\ref{mainthm}.
\end{rem}

\begin{rxm}\label{runex-trsp}
(a) Consider the fibration $\domF: \FamC \ftr \Set$ from
\rxx\ref{runex-fib}(a). 
And suppose that \ct{C} has a stable initial object,
\ie an initial object $0$ such that
$\xymatrix@1@=1.5em{
0\times A\ar[r]|-{\raisebox{.9ex}{\makebox[0pt]{$\scriptstyle\sim$}}}&0}$
for all $A$. 
Let $1$ be a terminal object of $\ct{C}$, and consider the family
$\eqobj X:X\times X\to\ct{C}$ as the function that maps $(a,b)$ to $1$
if $a=b$ and to $0$ otherwise.
There are two natural transformations
$\iota : \tt_X \tnat \eqobj X \pr{1,1}$,
whose component on $x \in X$ is the identity,
and $\tau_A : (A\pr1) \land \eqobj X \tnat A \pr2$
whose component on $(x_1,x_2)$ is the identity on $A(x_1)$ if
$x_1 = x_2$, and the unique arrow $0 \to A(x_2)$ otherwise.
The object $\eqobj X$ and arrows $(\pr{1,1},\iota), (\pr2,\tau_A)$ for $A$ over $X$,
form a strict \transporter for the set $X$. And these are productive as the components on
$(x_1,y_1,x_2,y_2)$ of $\eqobj X \glprd \eqobj Y$ and $\eqobj{X\times Y}$ are both initial or
both terminal.
Hence one may take the canonical iso as the component of $\chi_{X,Y}$ on $(x_1,y_1,x_2,y_2)$.
\smallskip

\noindent(b)
Let \ct{C} be a category with finite products and
suppose that \ct{C} has a weak factorisation system $(\leftcl,\rightcl)$ such that
\ct{C} has pullbacks of arrows in \rightcl along any arrow.
It follows that arrows in the right class \rightcl satisfy the hypothesis of \rxx\ref{runex-fib}(b),
so $\cod\rst{\rightcl} : \rightcl \ftr \ct{C}$ is a fibration with products.
If arrows in the left class \leftcl are stable under pullback
along arrows in the right class \rightcl,
then every object $X$ of \ct{C} has a strict \transporter defined as follows.
A loop $\cocD X\colon=\ple{\refl{X},\pr{1,1}}$ is obtained factoring
the diagonal $\pr{1,1} : X \to X \times X$ as 
\[\xymatrix@C=3em{X\ar[r]^-{\refl X}\ar[d]_-{\id{X}}&PX\ar[d]^-{\eqobj X}\\
X\ar[r]_-{\pr{1,1}}&X \times X}\]
where \refl{X} is in \leftcl.
For an arrow $a \in \rightcl$, consider the following commutative diagram.
\[\xymatrix@C=3em{
A	\ar@/_1em/[ddr]_(.65){\id A} \ar[rr]
	\ar[dr]|-{\strut\refl{a}\strut} \ar@/^2em/[rrrr]^-{\id A}
&&	X	\ar[dr]|-{\strut\refl{X}\strut}
		\ar@/_1em/|(.44){\hole}[ddr]_(.65){\id X}
&&	A	\ar[d]^-a
\\
&A\times_X P X\ar[d]\ar[rr]&&
P X\ar[d]^-{\pr1\eqobj X}\ar[r]^-{\pr2\eqobj X}&X
\\
&	A	\ar[rr]_-a	&&	X&}\]
Since the arrow \refl{a} is a pullback of \refl{X} along an arrow in \rightcl,
it is in \leftcl.
It follows by weak orthogonality that there is $t_a : A \times_X PX
\to A$ filling in the previous diagram
\[\xymatrix@C=3em{
A\ar[d]_-{\refl{a}}\ar[rr]^-{\id A}&&A\ar[d]^-a\\
A\times_XP X\ar[r]\ar[rru]|-{t_a}&P X\ar[r]^-{\pr2\eqobj X}&X}\]
A carrier at $a$ is then $(\pr{2},t_a)$.
Instances of this situation can be found
in any Quillen model category where acyclic cofibrations
are stable under pullback along fibrations, but also in
Shulman's type-theoretic fibration categories~\cite{Shulman2015} and
Joyal's tribes~\cite{Joyal2017}.

Suppose now that the class \leftcl is stable under products
in the sense that, for every object $X$ in \ct{C},
the functor $(\blank) \times X : \ct{C} \to \ct{C}$ maps \leftcl into \leftcl.
Then $\cod\rst{\rightcl}$ has strict productive \transporter{}s.
Indeed in this case
$\cocD X \glprd \cocD Y = (\pr{1,2,1,2},\refl{X} \times \refl{Y})$
and the arrow $\refl{X} \times \refl{Y}$ is in \leftcl
as it factors as shown below.
\[
\xymatrix@C=5em{
X \times Y 	\ar[r]^-{\refl X \times Y} \ar@/_1em/[rr]_-{\refl X \times \refl Y}
&	P X \times Y	\ar[r]^-{PX \times \refl Y}
&	PX \times PY
}
\]
And the rest of the argument is similar to the one in (c) below.
\smallskip

\noindent(c)
Let $(\ct{C},\weqcl,\fibrcl)$ be a path category.
It consists of two full subcategories \weqcl and \fibrcl of $\ct{C}\twr$
closed under isomorphisms and satisfying some additional conditions, see~\cite{vdBergMoerdijk2018}.
In particular, the category \fibrcl satisfies the hypothesis of \rxx\ref{runex-fib}(b),
so $\cod\rst{\fibrcl} : \fibrcl \ftr \ct{C}$ is a fibration with products.
In the notation of~\cite{vdBergMoerdijk2018},
the arrow $(s,t) : PX \to X \times X$ together with the arrow $r : X \to PX$ and,
for every $f \in \fibrcl$, a transport structure in the sense of~\cite[Def.\ 2.24]{vdBergMoerdijk2018},
provides a (not necessarily strict) \transporter for $X$.
In the following, when dealing with a path category,
we shall always try to stick to the notation in \cite{vdBergMoerdijk2018};
however we prefer to denote the arrow $r : X \to PX$ as $\refl X$.
Since the arrows in \weqcl are stable under pullback along arrows in \fibrcl,
see~\cite[Prop.\ 2.7]{vdBergMoerdijk2018},
it follows that $\refl X \times \refl Y$ is in \weqcl,
as terminal arrows are in \fibrcl.
Hence we obtain $\chi_{X,Y}$ with the required properties as the arrow $(\id{},k)$,
where $k : PX \times PY \to P(X \times Y)$ is a lower filler in
\[
\xymatrix{
X \times Y	\ar[d]_-{\refl X \times \refl Y} \ar[r]^-{\refl{X \times Y}}
&	P(X \times Y)	\ar[d]^{(s,t)}
\\
PX \times PY	\ar[r]	&	X \times Y \times X \times Y	}
\]
see~\cite[Lemma 2.9]{vdBergMoerdijk2018}.
It follows that the fibration $\cod\rst{\fibrcl}$ has productive \transporter{}s.
Note that these are not necessarily strict
as the lower filler need not make the upper triangle commute.
\end{rxm}

\begin{rem}
Example \ref{runex-trsp}(c) fits only momentarily in the framework that we are developing.
This will become clear after \thx\ref{mainthm},
as our aim is to characterise elementary fibrations.
This suggests the relevance of a weaker notion than elementary fibration,
which we shall consider in future work.
\end{rem}

\begin{rem}\label{prdcocD}
For $X,Y$ in \ct{B},
we can rewrite
$\cocD{X}\glprd\cocD{Y}:\top_X\to\eqobj{X}\boxtimes\eqobj{Y}$
as the composite $<\iota,\cocD Y'!>\cocD X'$ as shown in the diagram
below.
\[\xymatrix@C=5em@R=2em{
&&&\trmf Y\ar[r]^-{\cocD Y}
&\eqobj Y
\\
&&\trmf{X\times Y\times X}\xycar[ur]\ar@{-->}[r]^-{\cocD Y'}
&\ridx{\pr{2,4}}{\eqobj Y}\xycar[ur]
&\\
\trmf X\ar[dr]_-{\cocD X}
&\trmf{X\times Y}\xycar[l]\ar@{-->}[dr]_(.4){\cocD X'}\ar[r]^-{\cocD X'}
&\ridx{\pr{1,3}}{\eqobj X}\ar[u]_-{!}\ar[d]^-{\id{}}\ar[r]^-{<\iota,\cocD Y'!>}
&\xeqobj X Y\ar[u]_-{\pi_2}\ar[d]^-{\pi_1}
&\\
&\eqobj X&\ridx{\pr{1,3}}{\eqobj X}\xycar[l]\xycar[r]^-{\iota}
&\ridx{\pr{1,3}}{\eqobj X}
&\\
X\ar[dr]|-{\phantom{I}\pr{1,1}}
&X\!\times\!Y\ar[l]|-{\pr1}\ar[dr]|-{\phantom{I}\pr{1,2,1}}&&
Y\ar[r]|-{\pr{1,1}}&Y\!\times\!Y
\\
&X\!\times\!X
&X\!\times\!Y\!\times\!X\ar@(lu,ru)[]|{\id{}}
\ar[l]|-{\pr{1,3}}\ar[ur]|-{\,\pr 2\,}\ar[r]|-{\pr{1,2,3,2}}
&X\!\times\!Y\!\times\!X\!\times\!Y\ar[ur]|-{\,\pr{2,4}\,}
\ar@/^10pt/[l]|-{\pr{1,2,3}}\ar@/^21pt/[ll]|-{\pr{1,3}}}
\]
\end{rem}

\begin{dfnt}\label{def-partrsp}
Let $p : \ct{E} \ftr \ct{B}$ be a fibration with \transporter{}s.
Let $X,Y$ be in \ct{B} and let $A$ be over $Y \times X$ in \ct{E}.
A \dfn{parametrised carrier} at $A$ for the transporter on $X$ is
an arrow 
\[
\trspp Y A : \desdp X A \to A
\]
over $\pr{1,3} : Y \times X \times X \to Y \times X$.
We say that the parametrised carrier $ \trspp Y A $ is \dfn{strict} if
$\trspp Y A \cocaDp Y A = \id A$,
where $\cocaDp Y A : A \to \desdp X A$ is the arrow defined in \ntx\ref{bnota}.
\end{dfnt}

\begin{prp} \label{partrsp}
Let $p : \ct{E} \ftr \ct{B}$ be a fibration.
\begin{enumerate}[label=(\roman*)]
\item
If $p$ has productive \transporter{}s,
then for every $X,Y$ in \ct{B},
there is a parametrised carrier at every $A$ over $Y \times X$.
\item
If the productive \transporter{}s are strict,
then so are the parametrised carriers.
\end{enumerate}
\end{prp}

\begin{proof}
(i)
The arrow \trspp{Y}{A} can be obtained as
the composite
\[\xymatrix@C=3em{
\desdp X A	\ar[rr]^-{\alpha \land <\cocD Y' !, \iota >}
&&
(\ridx{\pr{1,2}} A) \land (\xeqobj Y X)	\ar[d]_-{\id{} \land \chi_{X,Y}}
&\\
&&	\desdg{Y \times X}{A}{\pr{1,2}}	\ar[r]^-{\trsp A}	&	A
\\
Y \times X \times X	\ar[rr]^-{\pr{1,2,1,3}}	&&%
Y \times X \times Y \times X	\ar[r]^-{\pr{3,4}}	&	Y \times X	}
\]
where $\alpha : \ridx{\pr{1,2}}A \car \ridx{\pr{1,2}} A $
and $\iota : \ridx{\pr{2,3}}(\eqobj X) \car \ridx{\pr{2,4}}(\eqobj X)$
are cartesian over $\pr{1,2,1,3}$ and
$\cocD Y':\trmf{Y\times X\times X}\to\ridx{\pr{1,3}}\eqobj Y $
is the unique arrow over $\pr{1,2,1,3}$ obtained by cartesianness from the composite
\[
\xymatrix@C=4em{
\trmf{Y\times X\times X}	\ar[r]^-{\pr1\carl{\trmf Y}}	&	\trmf Y	\ar[r]^-{\cocD Y}
&	\eqobj Y.}
\]

\noindent(ii)
Let $ \cocD X' : \trmf{Y \times X} \to \ridx{\pr{2,3}} \eqobj X $
be the unique arrow over $ \pr{2,3} $
obtained by cartesianness from $ \cocD X (\pr2 \carl{\trmf X}) $.
By \rmx\ref{prdcocD},
it is $ < \cocD Y' !, \iota > \cocD X' = \cocD Y \glprd \cocD X $.
It follows that $ \trspp Y A \cocaDp Y A = \trsp A \cocaD A = \id A $.
\end{proof}

\section{Elementary fibrations}\label{four}

Recall from \cite[3.4.1]{JacobsB:catltt} the following
definition.

\begin{dfnt}\label{elmfbr}
A fibration with products $ p : \ct{E} \ftr \ct{B} $ is \dfn{elementary} if,
for every pair of objects $ Y $ and $ X $ in \ct{B},
reindexing along the parametrised diagonal $ \pr{1,2,2} : Y \times X \to Y \times X \times X $
has a left adjoint $\ladjD{Y,X} : \ct{E}_{Y \times X} \ftr \ct{E}_{Y \times X \times X} $,
and these satisfy: 
\begin{description}
\item[Frobenius Reciprocity:]
for every $A$ over $Y\times X$ and $B$ over $Y\times X\times X$,
the canonical arrow
\[
\ladjD{Y,X}(\ridx{\pr{1,2,2}}B\land A)\ftr B\land\ladjD{Y,X}A
\]
is iso, and
\item[Beck-Chevalley Condition:]
for every pullback square
\[\xymatrix@C=6em@R=3em{
Y\times X\ar[d]_{\pr{1,2,2}}\ar[r]^-{f\times X}&%
Z\times X\ar[d]^{\pr{1,2,2}}\\
Y\times X\times X\ar[r]_-{f\times X\times X}&Z\times X\times X}\]
and every $A$ over $Z\times X$, the canonical arrow
\[
\ladjD{Y,X}\ridx{(f\times X)}A\ftr\ridx{(f\times X\times X)}{\,\ladjD{Z,X} A}
\]
is iso.
\end{description}
\end{dfnt}

Since the collection of parametrised diagonals will be often referred
to in the following, it is useful to introduce a notation for it.

\begin{nota}\label{bnota}
We write \gdiagcl for the class of arrows of the form
$\pr{1,2,2}:Y\times X\to Y\times X\times X$ in \ct{B}.
\end{nota}

\begin{rxm}\label{runex-sec}
(a)
The fibration $\domF:\FamC\ftr\Set$ of the
\rxx\ref{runex-fib} is elementary when $\ct{C}$ has finite
products and a stable initial object.
Indeed, let 
$\eqobj X:X\times X\to\ct{C}$ be the family defined in \rxx\ref{runex-trsp}.
Then, for every $A:Y\times X\to \ct{C}$, the family
\[
\xymatrix@R=1.5ex@C=1.8ex{
Y\times X\times X\ar[rr]^-{ \ladjD{Y,X}(A)}&&\ct{C}\\
(x,a,b)\ar@{|->}[rr]&&A(x,a)\times \eqobj X(a,b)
}
\]
determines the required left adjoint, see
\cite[3.4.3~(iii)]{JacobsB:catltt}.
\smallskip

\noindent (b) 
When $\ct{C}$ has finite limits, the fibration $\cod$ is elementary.%
\end{rxm}

Let $ p : \ct E \ftr \ct B $ be a functor.
For a class of arrows $ \Theta $ in \ct{B},
say that an arrow $ \phi $ in \ct{E} is \dfn{over $ \Theta $}
if $ p(\phi) \in \Theta $.
Recall from~\cite{Streicher:fibc} that 
an arrow $ \varphi : A \to B $ is \dfn{\fibepic{p}} if,
for every pair $ \psi,\psi' : B \to B' $ such that $ p(\psi) = p(\psi') $,
whenever $ \psi \varphi = \psi' \varphi $ it is already $  \psi = \psi' $.

\begin{rem}
\begin{enumerate}[label=(\roman*)]
\item
When $ p $ is a fibration, $\varphi$ is \fibepic{p} \iff
$\psi\varphi=\psi'\varphi$ implies $\psi=\psi'$
just for vertical arrows $\psi$ and $\psi'$.
\item
Every cocartesian arrow is \fibepic{p}.
\item
An arrow $\varphi : A \to B$ that factors as a cocartesian arrow
followed by a vertical $\upsilon$, is \fibepic{p} if and only if $\upsilon$ is \fibepic{p}.
Moreover, if $p$ is a fibration, this happens if and only if $\upsilon$ is epic in the fibre $\ct{E}_{p(B)}$.
\end{enumerate}
\end{rem}

Assume from now on that $ p : \ct{E} \ftr \ct{B}$ is a fibration.
We need to introduce a few definitions
that will be instrumental in formulating the main \thx\ref{mainthm}

It is well-known that,
whenever there is a commuting square in \ct{E}
with two opposite arrows cartesian and
sitting over a pullback square in \ct{B}
\begin{equation}\label{prd-stab}
\xycentre[@=3em]{
A'	\ar[d]_{\varphi'} \xycar[r]	&	A	\ar[d]^-{\varphi\quad}="E"
&&	W'	\ar[r] \ar[d]_-{\quad}="B"	&	W	\ar[d]
\\
B'	\xycar[r]	&	B
&&	Z'	\ar[r]	&	Z	\ar@{|->}"E";"B"^p	}
\end{equation}
then the left-hand square in~\eqref{prd-stab} is a pullback too.

We say that a class $ \Phi $ of arrows in \ct{E} is \dfn{\BCarL} when,
in every diagram \eqref{prd-stab} where the right-hand pullback is of the form
\[\xymatrix@=3em{
U\times Y\ar[r]^{g\times Y}\ar[d]_-{U\times f}&V\times Y\ar[d]^-{V\times f}\\
U\times X\ar[r]_-{g\times X}&V\times X}\]
and $\varphi$ is in $\Phi$, also $\varphi'$ is in $\Phi$.
In such a situation,
we may say that $\varphi'$ is a \dfn{\BCop of $\varphi$ along $g$}.

\begin{lem}\label{blem}
Suppose \ct{B} has binary products and $p:\ct{E}\ftr\ct{B}$ is a
fibration with left adjoints to reindexing along arrows in
\gdiagcl. Let $\Phi$ be the class of cocartesian lifts of arrows in
\gdiagcl, which exist thanks to \rmx\ref{brem}.
Then the following are equivalent:
\begin{enumerate}[label={\normalfont(\roman*)},ref={\normalfont(\roman*)}]
\item\label{uno} The class $\Phi$ is \BCarL.
\item\label{due} The left adjoints satisfy the Beck-Chevalley Condition.
\end{enumerate}  
\end{lem}

\begin{proof}
\ref{uno}\Implies\ref{due}
Given $f:Y\to Z$ and an object $A$ over $Z\times X$, in the
commutative diagram
\[\xymatrix{\ridx{(f\times X)}{A}\xycar[r]\ar@{-->}[d]&A\xycocar[d]\\
\ridx{(f\times X\times X)}{(\ladjD{Z,X}A)}\xycar[r]&\ladjD{Z,X}A}\]
the dotted arrow is cocartesian by \ref{uno}.
The statement follows by \rmx\ref{brem}.

\noindent\ref{due}\Implies\ref{uno} Consider a diagram like in
(\ref{prd-stab}) for $f$ in \gdiagcl and $\varphi$ cocartesian over
it:
\[\xymatrix@C=3em@R=3em{
A'\ar[d]_{\varphi'}\xycar[r]&A\xycocar[d]^-{\varphi\quad}="E"
&&U\times X\ar[rr]^{g\times X}\ar[d]_-{\ \pr{1,2,2}}="B"&&
V\times X\ar[d]^-{\pr{1,2,2}}\\
B'\xycar[r]&\strut B
&&U\times X\times X\ar[rr]_-{g\times X\times X}&&
V\times X\times X\ar@{|->}"E";"B"^-p}
\]
So, by \rmx\ref{brem}, it is $B\cong\ladjD{V,X}(A)$. Hence
$B'\cong\ladjD{U,X}(A')$  by \ref{due} which yields that also
$\varphi'$ is cocartesian, again by \rmx\ref{brem}.
\end{proof}

We say that $ \Phi $ is \dfn{\FRar}
if, for every $\varphi:A\to B$ and every cartesian arrow $\psi:C'\car
C$ over $p(\varphi)$,
the arrow $\varphi\land\psi \colon= < \varphi \pi_1, \psi \pi_2>  :A\land C'\to B\land C$ is in $\Phi$
whenever $\varphi$ is.
\[\xymatrix{
A\ar[d]_{\varphi}&A\land C'\ar[l]\ar[d]^{\varphi\land\psi}\ar[r]&
C'\xycar[d]^{\psi}\\
B&B\land C\ar[l]\ar[r]&C}\]

\begin{lem}\label{dlem}
Let $p:\ct{E}\ftr\ct{B}$ be a
fibration with finite products and
suppose that it has left adjoints to reindexing along arrows in \gdiagcl.
Let $\Phi$ be the class of cocartesian lifts of arrows in \gdiagcl,
which exist thanks to \rmx\ref{brem}.
Then the following are equivalent:
\begin{enumerate}[label={\normalfont(\roman*)},ref={\normalfont(\roman*)}]
\item\label{duno} The class $\Phi$ is \FRar.
\item\label{ddue} The left adjoints satisfy the Frobenius Reciprocity.
\end{enumerate}  
\end{lem}

\begin{proof}
For a parametrised diagonal
$\pr{1,2,2}:Y\times X\to Y\times X\times X$ and an object $A$ over
$Y\times X$, consider a commutative diagram
\[\xymatrix{A\xycocar[d]_{\varphi}&
A\land\ridx{\pr{1,2,2}}{(B)}\ar[l]\ar[d]^{}\ar[r]&
\ridx{\pr{1,2,2}}{(B)}\xycar[d]^{}\\
\ladjD{Y,X}(A)&\ladjD{Y,X}(A)\land B\ar[l]\ar[r]&B}\]
By \rmx\ref{brem} the middle arrow is cocartesian if and only if
$\ladjD{Y,X}(A)\land B\cong
\ladjD{Y,X}(A\land\ridx{\pr{1,2,2}}{(B)}$. Hence the statement
follows.
\end{proof}

Let $r:Y\to X$ and $s:X\to Y$ be a retraction pair in \ct{B} with
$rs=\id X$ and let
$\rho:C'\car C$ and $\sigma:C\car C'$ be cartesian over $r$ and $s$
respectively so that $\rho\sigma=\id{C}$.
Given $\varphi:A\to B$ over $s$,
we say that $ \sigma \land \varphi : C \land A \to C' \land B $
is a \dfn{\FRop of $ \varphi $ with $ C $}.

Given a retraction pair $ r : Y \to X $ and $s : X \to Y$ in \ct{B},
and an arrow $ \varphi : A \to B $ over $ V \times s : V \times X \to V \times Y $,
let $ \BCFRcl s r (\varphi) $ be
the class of all arrows obtained from $ \varphi $
by first a \BCop and then a \FRop.
More specifically, the class $ \BCFRcl s r (\varphi) $ consists of
those arrows $\psi$ fitting in a commutative diagram
\begin{equation}\label{BCFR}
\xycentre{
&&A\ar[rr]^{\varphi}&&B&&\\
&A'\xycar[ur]\ar[rr]^{\varphi'}&&B'\xycar[ur]&&&\\
\ar@{|->}[ddd]_p&
C\land A'\ar[u]^{\pi_2}\ar[d]_{\pi_1}\ar[rr]^{\psi}&&
C'\land B'\ar[u]_{\pi_2}\ar[d]^{\pi_1}&&&\\
&C\xycar[rr]^-{\sigma}\ar@<-1ex>@/_1.5em/[rrrr]|-{\,\id\,}
&&C'\xycar[rr]^-{\rho}
&&C&\\
&&&&&&\\
&&V\times X\ar[rr]_{V\times s}
&&V\times Y
\\
&U\times X\ar[ur]|-{\,g\times X\,}
\ar[rr]^{U\times s}\ar@/_1.5em/[rrrr]|-{\,\id\,}&&
U\times Y\ar[ur]|-{\,g\times Y\,}\ar[rr]^{U\times r}
&&U\times X
}
\end{equation}
In the following, we shall only be concerned with those classes of arrows $\BCFRcl{s}{r}(\phi)$
for which $r = \pr{1}: X \times X \to X$ and $s = \pr{1,1}: X \to X \times X$,
and $\phi$ is over the class \gdiagcl of parametrised diagonals.

\begin{rem}
Note that the class $\BCFRcl s r (\varphi) $ is closed under isomorphism.
\end{rem}

We are at last in a position to state the main result of the paper.

\begin{thm} \label{mainthm}
Let $p:\ct{E}\ftr\ct{B}$ be a fibration with products.
The following are equivalent:
\begin{enumerate}[label={\normalfont(\roman*)},ref={\normalfont(\roman*)}]
\item	\label{E1}
The fibration $p:\ct{E}\ftr\ct{B}$ is elementary.
\item	\label{E2}
\begin{innerenu}{\ref{E2}}
\item	\label{E2coca}
Every arrow in \gdiagcl has all cocartesian lifts.
\item	\label{E2BCFR}
The cocartesian arrows over \gdiagcl are \BCar and \FRar.
\end{innerenu}
\item	\label{E3}
\begin{innerenu}{\ref{E3}}
\item	\label{E3coca}
For every object $ X $ in \ct{B} there is
an object $ \eqobj X $ over $X\times X$
and a cocartesian arrow $ \cocD X : \trmf X \to \eqobj X $
over $ \pr{1,1}:X\to X\times X$.
\item	\label{E3BCFR}
The cocartesian arrows over \gdiagcl are \BCar and \FRar.
\end{innerenu}
\item	\label{E4}
\begin{innerenu}{\ref{E4}}
\item	\label{E4prdtrsp}
The fibration $p$ has strict productive transporters.
\item	\label{E4pepi}
Let $\cocD X : \tt_X \to \eqobj X$ denote
the loop for the transporter on $X$.
Every arrow in $\cocaDpcl X$ is \fibepic{p} over \gdiagcl.
\end{innerenu}
\item\label{E5}
\begin{innerenu}{\ref{E5}}
\item\label{E5eq}
For every $X$ in \ct{B} there are
an object $\eqobj X$ over $X\times X$
and an arrow $\cocD X:\trmf X\to\eqobj X$
over $\pr{1,1}:X\to X\times X$.
\item\label{E5coca}
For every object $X$ in \ct{B},
the arrows in $\cocaDpcl X$ are cocartesian over \gdiagcl.
\end{innerenu}
\item\label{E6}
For every $ X $ in \ct{B} there is an object \eqobj{X} over
$X\times X$ such that,
for every $Y,X \in \ct B $ and every $ A $ over $ Y \times X $,
the assignment 
\[\xymatrix{
A\ar@{|->}[r]&\desdp{X\times X} A}\]
gives rise to a left adjoint to the reindexing functor
$\ridx{\pr{1,2,2}}:\ct{E}_{Y\times X\times X}\ftr\ct{E}_{Y\times X}$.
\end{enumerate}
\end{thm}

\begin{proof}
\ref{E1}\Iff\ref{E2}
By \rmx\ref{brem}, the equivalence follows from \lmx\ref{blem} and
\lmx\ref{dlem}.

\noindent\ref{E2}\Implies\ref{E3}
The object $\eqobj X$ over $X \times X$ and the arrow $\cocD X$
are obtained taking a cocartesian lift of $\pr{1,1}:X\to X\times X$
from \trmf X.

\noindent\ref{E3}\Implies\ref{E4}
We begin constructing a strict \transporter on an object $X$.
To this aim,
it is enough to construct a carrier \trsp{A} for $A$ over $X$.
Note that the arrow \cocaD{A} from \ntx\ref{nota-cocar} is
obtained from \cocD{X} by \FRop, hence \cocaD{A} is in \cocaDpcl{X},
therefore cocartesian.
The universal property of
\cocaD{A} yields a unique arrow \trsp{A}
as in the diagram
\[\xymatrix{
A	\ar[dr]_{\id A}\xycocar[rr]^{\cocaD A}
&&	\desd X A	\ar@{-->}[dl]^{\trsp A}	\ar@{}[d]^(.7){\qquad}="E"
&&	X\ar@{}[d]_(.7){\quad}="B"\ar[dr]_{\pr 1}\ar[rr]^{\pr{1,1}}
&&	X\times X\ar[dl]^{\pr 2}
\\
&A	&&&&	X	&	\ar@{|->}"E";"B"_-p	}\]
To prove that this choice of \transporter{}s is productive,
let $ X,Y \in \ct B $.
We can rewrite  the arrow
$\cocD X\glprd\cocD Y : \top_X \to \eqobj{X}\boxtimes\eqobj{Y}$ as the
composite 
$<\sigma,\cocD Y'!>\cocD X'$ with the notation in \rmx\ref{prdcocD}.
The diagram therein also shows that $\cocD X'$ and $<\sigma,\cocD Y'!>$
are in $\cocaDpcl X$ and $\cocaDpcl Y$,
respectively.
It follows that each is cocartesian,
thus so is $ \cocD X \glprd \cocD Y $. 
Its universal property applied to $ \cocD{X \times Y} $
then yields the required $ \chi_{X,Y} $.

To prove condition \ref{E4pepi},
recall that arrows in \cocaDpcl X are obtained from \cocD{X}
first by \BCop and then with a \FRop,
as in diagram \eqref{BCFR}.
Since \cocD{X} is a cocartesian arrow over \gdiagcl
and these are \BCar and \FRar,
arrows in $ \cocaDpcl X $ are cocartesian, in particular \fibepic{p}.

\noindent\ref{E4}\Implies\ref{E5}
There is only condition \ref{E5coca} to prove; we need to show that, given $X$
in \ct{B}, arrows in $\cocaDpcl X$ are cocartesian.
These are obtained by \BCop of \cocD{X} along any projection
$\pr{2,3}:Y\times X\times X\to X\times X$, and then by \FRop.
Hence, for every $A$ over $Y\times X$,
there is exactly one arrow $A\to\desdp XA$ in \cocaDpcl{X} which is
$\cocaDp YA\colon=<\pr{1,2,2}\carl{(\ridx{\pr{1,2}} A)},\delta>$
introduced in \ntx\ref{bnota}.
Let $ \phi : A \to B $ be an arrow over $ \pr{1,2,2} : Y \times X \to Y \times X \times X $
and consider the following diagram
\begin{equation}\label{cocaDp}
\xycentre[@C=1em]{
A	\ar[d]_{\vrt\phi} \ar@/^/[ddr]^(.75){\phi} \ar[rr]^-{\cocaDp Y A}	&&%
	\desdp X A	\ar[d]^{(\ridx{\pr{1,2}} \vrt\phi) \land \id{}}	&\\
\ridx{\pr{1,2,2}} B	\xycar[dr]_-{\beta} \ar[rr]^(.45){\cocaD{(\ridx{\pr{1,2,2}} B)}}|(.25){\hole}	&&%
	\desdgp{X}{B}{\pr{2,3}}{\pr{1,2,2}'}	\xycar[dr]^-{\gamma}	&\\
&	B	&&	\desdgp{X}{B}{\pr{3,4}}{\pr{1,2,3}}	\ar[ll]_-{\trspp{Y \times X} B}	\\
Y \!\times\! X	\ar[dr]_-{\pr{1,2,2}} \ar[rr]^-{\pr{1,2,2}}	&&%
	Y \!\times\! X \!\times\! X	\ar[dr]^-{\pr{1,2,2,3}} \ar[dl]|-{\,\id{}\,}	&\\
&	Y \!\times\! X \!\times\! X	&&	Y \!\times\! X \!\times\! X \!\times\! X,	\ar[ll]^-{\pr{1,2,4}}	}
\end{equation}
where $ \pr{1,2,2}' = \pr{1,2,2} \pr{1,2} : Y \times X \times X \to Y \times X \to Y \times X \times X $,
the arrow \trspp{Y \times X}{B} is a strict parametrised carrier at $B$ obtained by \prx\ref{partrsp}
and $ \beta $ and $ \gamma $ are cartesian.
Commutativity of the upper square with vertex $ B $ follows from
$ \trspp{Y \times X} B \cocaDp{Y \times X} B = \id B $ and commutativity of the diagram below.
\[\xymatrix@R=1em@C=1.8em{
&	\trmf X	\ar[rrr]^{\cocD X}	&&&	\eqobj X
\\
\trmf{Y \times X}	\ar[rrr]|!{[dr];[ur]}{\hole} \xycar[dr] \xycar[ur]
&&&	\ridx{\pr{2,3}} \eqobj X	\xycar[dr] \xycar[ur]
&\\
&	\trmf{Y \times X \times X}	 \ar[rrr] \xycar[uu]
&&&	\ridx{\pr{3,4}} \eqobj X	\xycar[uu]
\\
\ridx{\pr{1,2,2}} B	\xycar[dr]|-{\,\beta\,}%
\ar[rrr]|!{[dr];[ur]}{\hole}|(.55){\,\cocaD{(\ridx{\pr{1,2,2}} B)}\,}%
\ar[uu] \xycar@/_2em/[ddrrr]|(.6){\hole\hole}
&&&	\desdgp{X}{B}{\pr{2,3}}{\pr{1,2,2}'}	\ar[dd]|-{\hole} \ar[uu]|-{\hole} \xycar[dr]|-{\,\gamma\,}
&\\
&	B	 \ar[rrr]|(.35){\,\cocaDp{Y \times X} B\,} \ar[uu] \xycar@/_2.5em/[ddrrr]	&&&%
	\desdgp{X}{B}{\pr{3,4}}{\pr{1,2,3}}	\ar[dd] \ar[uu]
\\
&&&	\ridx{\pr{1,2,2}} B	\xycar[dr]
&\\
&&&&	\ridx{\pr{1,2,3}} B
\\
&&&&\\
&	X	\ar[rrr]^{\pr{1,1}}	&&&	X \!\times\! X
\\
&&&&\\
Y \!\times\! X	\ar[rrr]|!{[dr];[ur]}{\hole}_(.6){\pr{1,2,2}} \ar[dr]_-{\,\pr{1,2,2}\,} \ar[uur]^{\pr2}
&&&	Y \!\times\! X \!\times\! X	\ar[dr]|{\pr{1,2,2,3}} \ar[uur]|{\pr{2,3}}
&\\
&	Y  \!\times\! X \!\times\! X	\ar[rrr]_-{\pr{1,2,3,3}} \ar[uuu]_(.65){\pr3}	&&&	Y \!\times\! X \!\times\! X \!\times\! X	 \ar[uuu]_{\pr{3,4}}
}
\]
Therefore diagram \eqref{cocaDp} commutes.
Since \cocaDp{Y}{A} is \fibepic{p}, it follows that there is exactly one vertical arrow
$ \phi' : \desdp X A \to B $
such that $ \phi' \cocaDp Y A = \phi $.
Hence \cocaDp{Y}{A} is cocartesian.

\noindent\ref{E5}\Iff\ref{E6}
This is just an instance of the equivalence in \rmx\ref{brem}.

\noindent\ref{E6}\Implies\ref{E1}
It is straightforward to verify that the left adjoints specified in \ref{E6}
satisfy the Beck-Chevalley condition and Frobenius reciprocity.
For the latter, a useful remark is that
the left adjoint to reindexing along 
$\pr{1,2,1,2}:X\times X\to X\times X\times X\times X$
\[\xymatrix{
A\ar@{|->}[r]&\ridx{\pr{1,2}}{A}\land \eqobj{X\times X}}\]
is isomorphic to
$\xymatrix@1{
A\ar@{|->}[r]&\ridx{\pr{1,2}}{A}\land (\eqobj{X}\boxtimes \eqobj{X})}$
as it follows from the fact that the diagram
\[\xymatrix@C=5em{
X\times X	\ar[d]_-{\pr{1,2,1,2}} \ar[r]_-{\pr{1,2,2}}
&	X\times X\times X	\ar[r]^-{\iso}_-{\pr{2,3,1}}
&	X\times X\times X	\ar[d]^-{\pr{1,2,3,3}}
\\
	X\times X\times X\times X
&&	X\times X\times X\times X	\ar[ll]_-{\iso}^-{\pr{3,1,4,2}}}\]
commutes.
\end{proof}

\begin{rem}
In an elementary fibration $p$,
the canonical arrow
$\omega_{X,Y} : \eqobj{X \times Y} \to \eqobj{X} \glprd \eqobj{Y}$ of
\rmx\ref{mah} is the inverse of $\chi_{X \times Y}$.
\end{rem}

\begin{rem}
Since faithful fibrations are equivalent to indexed posets,
the equivalence between condition \ref{E1} and condition \ref{E4} in
\thx\ref{mainthm} gives Proposition 2.4 of \cite{EmmeneggerJ:eledac}.
\end{rem}

\begin{prp}\label{cob}
If $p:\ct{E}\ftr\ct{B}$ is an elementary fibration, \ct{A} is a
category with finite products, and $F:\ct{A}\ftr\ct{B}$ is a
functor which preserves finite products, then the 
fibration $F^*p:F^*\ct{E}\ftr\ct{A}$ is also elementary.
\end{prp}

\begin{proof}
Since $F$ preserves finite products, the fibration $F^*p$ has finite
products. To see that $F^*p$ is elementary, we apply
\thx\ref{mainthm}\ref{E4}. For $V$ an object in \ct{C}, a \transporter 
on $F(V)$ for the fibration $p$ is also a \transporter on $V$ for the
fibration $F^*p$ since
$\xymatrix@1{\ple{F(\pr1),F(\pr2)}:F(V\times V)\ar[r]^-{\iso}
&FV\times FV}$. Condition \ref{E4pepi} for $F^*p$ follows
directly from the same condition for $p$ since the arrows in
\cocaDpcl{V} are the arrows in
\BCFRcl{F(\pr{1,1})}{F(\pr1)}{(\cocD{FV})}.
\end{proof}

\section{Elementary fibrations from \awfs}
\label{Applications}

Consider the fibration $\cod\rst{\rightcl} : \rightcl \ftr \ct{C}$
associated to the subcategory \rightcl from
a weak factorisation system $(\leftcl,\rightcl)$
in \rxx\ref{runex-fib}(b).
As we saw in \xsx\ref{runex-trsp}(b),
the fibration $\cod\rst{\rightcl}$ has strict productive transporters
when the base category \ct{C} has finite products, \leftcl is closed
under products and closed under pullbacks along arrows in \rightcl. 
In this case, for every $(f_0,f_1) \in \cocaDpcl{X}$,
the arrow $f_1$ is in $\leftcl$.
Indeed, \refl{X} is in \leftcl by construction and, for every
$\cocaDp{Y}{A} = (\pr{1,2,2},\ple{\id{A},\refl{X} \pr2 a})$,
the arrow $\ple{\id{A},\refl{X} \pr2 a}$ is a pullback along an arrow in \rightcl
of a product of \refl{X} as in the diagram
\[\xymatrix@C=3em@R=3em{
A	\ar[d]_-a \ar[rr]_-{\ple{\id{A},\refl{X} \pr2 a}} \ar@/^2em/[rrrr]^-{\id A}
&&	A \times_X PX	\ar[d] \ar[rr]
&&	A	\ar[d]^-{a}
\\
Y \times X	\ar[rr]^-{Y \times \refl{X}} \ar@/_2em/[rrrr]_-{\id{Y \times X}}
&&	Y \times PX	\ar[r]
&	Y \times X \times X	\ar[r]^-{\pr{1,2}}	&	Y \times X
}\]
where the right-hand square is a pullback.

\begin{lem}\label{lem-strfs}
Let \arrcl{F} be a full subcategory of $\ct{C}\twr$ closed under
pullbacks. Given an arrow
in \arrcl{F}
\[\xymatrix@C=4em{
A\ar[d]_-a\ar[r]_-{f_1}&B\ar[d]^-b\\
X\ar[r]_-{f_0}&Y,}\]
the following are equivalent:
\begin{enumerate}[label={\normalfont(\roman*)}]
\item
The arrow $(f_0,f_1)$ is \fibepic{\cod\rst{\arrcl{F}}}.
\item
Every left lifting problem for $f_1$ against arrows in \arrcl{F} has
at most one solution.
\end{enumerate}
\end{lem}

\begin{proof}
(i)$\Implies$(ii)
It is enough to show that a lifting problem of the form
\[\xymatrix{
A	\ar[d]_-{f_1} \ar[r]^-h	&	C	\ar[d]^-{c}
\\
B	\ar[r]_{\id{B}}	&	B
}\]
for $c\in\arrcl{F}$ has at most one solution.
Let then $g,g' : B \to C$ be two diagonal fillers.
They fit in the diagram below
which commutes except for the two parallel arrows.
\[\xymatrix@C=4em{
A	\ar[d]_-a \ar[r]_-{f_1} \ar@/^2em/[rr]^-h
&	B	\ar[d]_-b \ar@<.5ex>[r]^-g \ar@<-.5ex>[r]_-{g'}
&	C	\ar[d]^-{b c}
\\
X	\ar[r]_-{f_0}	&	Y	\ar[r]_-{\id{Y}}	&	Y
}\]
But, since by hypothesis $(f_0,f_1)$ is \fibepic{p}, also $g=g'$.

\noindent (ii)$\Implies$(i)
Let $(k,g),(k,g') : b \to c$ be such that $(f_0,f_1)(k,g) = (f_0,f_1)(k,g')$.
Then the commutative diagram
\[\xymatrix@=3em{
A	\ar[d]_-{f_1} \ar[r]^-h	&	C	\ar[d]^-c
\\
B	\ar[r]_-{kb} \ar@<.5ex>[ur]^-g \ar@<-.5ex>[ur]_-{g'}
&	Z
}\]
exhibits $g$ and $g'$ as solutions to a lifting problem for $f_1$.
It follows that $g = g'$.
\end{proof}

Combining \lmx\ref{lem-strfs} with \thx\ref{mainthm}
we immediately obtain a sufficient condition for the fibration
$\cod\rst{\rightcl}$ to be elementary.

\begin{cor}\label{cor-strfs}
Let $(\leftcl,\rightcl)$ be a weak factorisation system on a category
\ct{C} with finite products such that \leftcl is closed under products
and under pullbacks along arrows in \rightcl.
If $(\leftcl,\rightcl)$ is an orthogonal factorisation system
\ie diagonal fillers are unique,
then the fibration $\cod\rst{\rightcl}$ is elementary.
\end{cor}

\begin{rem}\label{intror}
\lmx\ref{lem-strfs}
provides also a way to find examples of non-elementary
fibrations among those of the form $\cod\rst{\rightcl}$
for a weak factorisation system $(\leftcl,\rightcl)$
as in \crx\ref{cor-strfs} and \xmx\ref{runex-trsp}(b).
Indeed, a cocartesian arrow $(\pr{1,1},r):\id{X}\to\eqobj{X}$
would give rise to a factorisation of the diagonal
$\pr{1,1}:X \to X \times X$
as the arrow $r$ followed by $\eqobj{X} \in \rightcl$.
It follows that $r$ is a retract of $\refl{X}$,
and thus it is in \leftcl.
As $(\pr{1,1},r)$ is in particular \fibepic{\cod\rst{\rightcl}},
the arrow $r$ would have unique solutions to lifting problems.
Hence, as soon as there is an object $X$ such that
no \leftcl-map in a factorisation of $\pr{1,1}:X \to X \times X$
has unique solutions to lifting problems,
$\cod\rst{\rightcl}$ cannot be elementary.
\end{rem}

The situation is however different if,
instead of looking at the fibration associated to
the full subcategory \rightcl of $\ct{C}\twr$,
one looks at the fibration associated to a non-full subcategory of
\rightcl. We present one such situation.

Recall from~\cite{GrandisTholen}, see also~\cite{BourkeGarner},
that an \dfn{\awfs} $(\awfsL,\awfsM,\awfsR)$
on a category \ct{C} consists of
functors $\awfsM:\ct{C}\twr\ftr\ct{C}$,
$\awfsR:\ct{C}\twr\to\ct{C}\twr$, and
$\awfsL:\ct{C}\twr\to\ct{C}\twr$
giving rise to a functorial factorisation %$f = (\awfsR f)(\awfsL f)$,
\[
\xymatrix@C=4em{
A	\ar[r]^-{\awfsL f} \ar@/_1em/[rr]_-{f}
&	\awfsM f	\ar[r]^-{\awfsR f}
&	B,	}
\]
and suitable monad and comonad structures on \awfsR and \awfsL
respectively, together with a distributivity law between them.
Let \RAlg be the category of algebras for the monad on \awfsR
and let \RMap be the category of algebras for the pointed endofunctor
on \awfsR. Similarly, let \LcAlg and \LMap be the categories of
coalgebras for the comonad on \awfsL and the pointed endofunctor on
\awfsL, respectively. When \ct{C} has finite limits, the two forgetful
functors
\[\xymatrix@1@C=3em{\RAlg\ar[rd]_(.45){\FTR{S}}\ar[r]^-{U'}&
\RMap\ar[d]^-{\FTR{N}}\ar[r]^-U&
\ct{C}\twr\ar[ld]^(.45){\cod}\\
&\ct{C}}\]
are homomorphisms of fibrations with finite products.

The category \Gpd of groupoids admits an \awfs
$(\awfsLg,\awfsMg,\awfsRg)$ 
such that the fibrations of \RgAlg and \RgMap over \Gpd are equivalent 
to the fibrations of split cloven isofibrations and of
normal cloven isofibrations, respectively, with arrows the commutative
squares that preserve the cleavage strictly,
see~\cite[Section 3]{GambinoLarrea2019}. 
Similar calculations show that also the category of small categories \Cat
admits an \awfs $(\awfsLc,\awfsMc,\awfsRc)$
such that the fibrations of \RcAlg and \RcMap over \Cat are equivalent
to the categories of split cloven isofibrations and of
normal cloven isofibrations, respectively, each with arrows the
homomorphisms of fibrations that preserve the cleavage on the nose.

The underlying weak factorisation systems
$(\leftcl,\rightcl)$ and $(\leftcl',\rightcl')$
of the two \awfs{s} are the (acyclic cofibrations, fibrations) weak
factorisation systems 
of the canonical, aka ``folk'', Quillen model structures on
\Cat and \Gpd, respectively.
As discussed in \rmx\ref{intror}, the associated fibrations
$\cod\rst{\rightcl}$ and $\cod\rst{\rightcl'}$
are not elementary.
On the other hand, the fibration
$\FTR{S}:\RcAlg\ftr\Cat$ is elementary as we shall see promptly.
From this it will also follow that the fibration
$\RgAlg\ftr\Gpd$ is elementary, and we shall be able to comment about
the other two obtained with the pointed endofunctor.

For a functor $F:\sct{A}\to\sct{B}$ between small
categories, $\awfsMc F$ is the category
whose objects are pairs $(A,x:B\isoa FA)$
where $A$ is an object in \sct{A} and $x$ is an iso in \sct{B},
and whose arrows $(b,a):(A,x:B\isoa FA)\to(A',x':B'\isoa FA')$
are pairs of an arrow $b:B\to B'$ in \sct{B} and an arrow $a:A \to A'$
in \sct{A} such that the square
\[\xymatrix{B\ar[d]_-b\ar[r]^-x&FA\ar[d]^-{Fa}\\
B'\ar[r]^-{x'}&FA'}\]
commutes.
Denote $\mathsf{i}_{\sct{B}}:\isosct{B}\to\sct{B}\twr$ the embedding
of the full subcategory \isosct{B} of $\sct{B}\twr$ on the isos.
Write 
\[\xymatrix@C=4em@R=3em{
\isosct{B}	\ar@<.5ex>[rr]^-{\mathsf{i}_{\sct{B}}}
		\ar@<-1.5ex>[dr]_-{\sdom{B}}
		\ar@<1.1ex>[dr]^-{\scod{B}}
&&	\sct{B}\twr	\ar@<-1.3ex>[ld]_-{\dom}
			\ar@<1.5ex>[ld]^-{\cod}
\\
&	\sct{B}	\ar@<-.2ex>[ur]|-{\idd{}\strut}
		\ar@<.2ex>[ul]|-{\strut\srefl{B}}
&}\]
the restrictions to \isosct{B} of the three structural functors.
Note that $\awfsMc F$ appears in the pullback of categories and functors
\begin{equation}\label{pulm}
\vcenter{\xymatrix@C=4em@R=2em{\awfsMc F\ar[r]^-{F'}\ar[d]_-{\scod{B}'}
&\isosct{B}\ar[d]^{\scod{B}}\\
\sct{A}\ar[r]^-{F}&\sct{B}}}
\end{equation}
and the functorial factorisation is obtained directly from it:
\[\xymatrix@C=4em@R=2em{
\sct{A}\ar@/_7pt/[rdd]_-{\Id{\sct{A}}}\ar[r]^-{F}
\ar[rd]|-(.6){\strut\awfsLc F}
&\sct{B}\ar[rd]^(.55){\srefl{B}}\ar@/^10pt/[rrd]^-{\Id{\sct{B}}}\\
&\awfsMc F\ar[r]^-{F'}\ar[d]_(.36){\,\scod{B}'}
\ar@/_20pt/[rr]|!{[r];[dr]}{\hole}_(.77){\awfsRc F}
&\isosct{B}\ar[d]_(.36){\scod{B}}\ar[r]^(.45){\sdom{B}}&\sct{B}\\
&\sct{A}\ar[r]^-{F}&\sct{B}.}\]
The factorisation extends to an \awfs on \Cat:
for the comonad the component at $F$ of the counit  is
\[\xymatrix{\sct{A}\ar[d]_-{\awfsLc F}\ar[r]^-{\Id{\sct{A}}}&
\sct{A}\ar[d]^-{F}\\
\awfsMc F\ar[r]^-{\awfsRc F}&\sct{B}}\]
while (the bottom component of) that of the comultiplication
$\awfsLc\tnat\awfsLc\awfsLc$ is
\[\Delta_F(A,x:B\isoa FA)=
(A,(x,\id{A}):(A,(A,x:B\isoa FA))\isoa(\awfsLc F)A)\]
---from here onward we leave out the definition of a functor on
arrows when it is obvious.
The component at $F$ of the unit of the monad is
\[\xymatrix{\sct{A}\ar[r]^-{\awfsLc F}\ar[d]_-{F}&
\awfsMc F\ar[d]^-{\awfsRc F}\\
\sct{B}\ar[r]^-{\sid{B}}&\sct{B}}\]
and (the top component of) that of the
multiplication
$\awfsRc\awfsRc\tnat\awfsRc$ is 
\[\mu_F((A,x:B\isoa FA),x':B'\isoa B)=(A,xx':B'\isoa FA).\]
The required distributivity law follows from the identities
\[(\awfsRc \awfsLc F)\circ\Delta_F = \id{\awfsMc F}
= \mu_F\circ(\awfsLc \awfsRc F),\]
see~\cite[2.2]{BourkeGarner}.

\begin{prp}\label{prel}
The fibration $\FTR{S}:\RcAlg\ftr\Cat$ is elementary.
\end{prp}

\begin{proof}
We shall make good use \thx\ref{mainthm} checking that the fibration
$\FTR{S}$ verifies condition \ref{E4}.
To construct a transporter on the small category \sct{B}
 consider first the functor 
$\ple{\scod{B},\sdom{B}}:\isosct{B}\to\sct{B}\times\sct{B}$
together with the structure  map $\sstr{B}$
defined by\footnote{This choice provides a stable functorial choice
of path objects in the sense
of~\cite[Definition~2.8]{GambinoLarrea2019}.}
\[\xymatrix@C=4em@R=.3ex{
(y:B_1\xyisoa B_2,(b_2,b_1):(B'_2,B'_1)\xyisoa(B_2,B_1))
\ar@{|->}[r]
&B'_1 \xrightarrow{b_2^{-1}yb_1} B'_2
\\
\awfsMc\ple{\scod{B},\sdom{B}}
\ar[r]^-{\sstr{B}}\ar[dddd]_{\awfsRc{\ple{\scod{B},\sdom{B}}}}
&\isosct{B}\ar[dddd]^{\ple{\scod{B},\sdom{B}}}\\ \\ \\ \\
\sct{B}\times\sct{B}\ar[r]^{\Id{\sct{B}\times\sct{B}}}&
\sct{B}\times\sct{B}.
}\]
Then to provide a loop on \sct{B} it is enough to show that the pair
$(\pr{1,1},\srefl{B})$ is a morphism from the algebra
$(\sid{B},\awfsRc\sid{B})$ to the algebra
$(\ple{\scod{B},\sdom{B}},\sstr{B})$, which is an easy diagram chase
in
\[\xymatrix@C=6em{
\awfsMc\sid{B}	\ar[d]_-{\awfsRc \sid{B}}
				\ar[r]^-{\awfsMc(\pr{1,1},\srefl{B})}
&	\awfsMc\ple{\scod{B},\sdom{B}}	\ar[d]^-{\sstr{B}}
\\
\sct{B}	\ar[r]^-{\srefl{B}}	&	\isosct{B}.
}\]
The construction of carriers is postponed to \lmx\ref{plem}.
But note that, once carriers are determined, transporters will be
strictly productive as the iso
\[\ple{\isosct{\pr{\mathrm{1}}},\isosct{\pr{\mathrm{2}}}}:
\isosct{B \times C} \cong \isosct{B} \times \isosct{C}\]
is clearly a morphism of algebras.

Finally, to see that morphisms in
$\BCFRcl{\pr{1,1}}{\pr 1}(\pr{1,1},\srefl{B})$ are \fibepic{\FTR{S}},
consider an algebra
$(\sct{A}\xrightarrow{F}\sct{I\times B},S)$;
write $D:\sct{A}\times_{\sct{B}}\isosct{B}\to\sct{I\times B \times B}$
for the underlying functor of
$(\ridx{\pr{1,2}}F) \land (\ridx{\pr{2,3}}\ple{\scod{B},\sdom{B}})$
and let $T:\awfsMc D \to \sct{A}\times_{\sct{B}}\isosct{B}$ be its
structure map,
which maps an object $((A,x),(i,b_2,b_1):(I,B_2,B_1)\isoa(FA,B))$
to $(S(A,(i,b_2)),b_2^{-1}xb_1)$.
Note that there is a functor
$K:\sct{A}\times_{\sct{B}}\isosct{B}\to\awfsMc(\pr{1,2,2}F)$
mapping an iso $x:B\isoa\pr2FA$ to
\[
(A,(\id{FA},x):(FA,B)\isoa\pr{1,2,2}FA)
\]
and that the composite
$\awfsMc(\sid{I\times B\times B},\ple{\sid{A},\srefl{B}\pr2F})K:
\sct{A}\times_{\sct{B}}\isosct{B}\to\awfsMc D$,
is a section of the algebra structure map.
Then for every
vertical morphism
$G:(\ridx{\pr{1,2}}F) \land (\ridx{\pr{2,3}}\ple{\scod{B},\sdom{B}})\to(F,S)$,
it is
\[
G = G T \awfsMc(\sid{I\times B\times B},\ple{\sid{A},\srefl{B}\pr2F}) K
= S \awfsMc(\sid{I\times B\times B},G\ple{\sid{A},\srefl{B}\pr2F})K.
\]
As $\cocaDp{\sct{I}}{F} = (\pr{1,2,2},\ple{\sid{A},\srefl{B}\pr2F})$,
algebra morphisms out of $(\ple{\scod{B},\sdom{B}},\sstr{B})$
are determined by their precomposition with \cocaDp{\sct{I}}{F}.
\end{proof}

\begin{cor}
The fibration $\RgAlg\ftr\Gpd$ is elementary.
\end{cor}

\begin{proof}
The algebraic weak factorisation system structure on \Gpd is obtained
pulling back that on \Cat along the embedding $\Gpd\ftr\Cat$. It
follows that $\RgAlg\ftr\Gpd$ is a change of base of $\RcAlg\ftr\Cat$
along the embedding $\Gpd\ftr\Cat$. Hence $\RgAlg\ftr\Gpd$ is
elementary by \prx\ref{cob}.
\end{proof}

\begin{lem}\label{plem}
Given an algebra $(F:\sct{A}\to\sct{B},S:\awfsMc F\to\sct{A})$ in
\RcAlg, there is exactly one carrier for the loop 
$(\pr{1,1},\srefl{B}):
(\sid{B},\awfsRc\sid{B})\to(\ple{\scod{B},\sdom{B}},\sstr{B})$ and
it is
$(\pr2,S):\ridx{\pr1}(F,S)\land(\ple{\scod{B},\sdom{B}},\sstr{B})\to F$.
\end{lem}

\begin{proof}
We may assume, without loss of generality, that 
the underlying functor of the algebra
$\ridx{\pr1}(F,S)\land(\ple{\scod{B},\sdom{B}},\sstr{B})$
is the diagonal $D: \awfsMc F \to \sct{B} \times\sct{B}$
in the pullback of categories and functors
\[\xymatrix@C=4em@R=3em{
\awfsMc F\ar[r]^{F'}	\ar[d]_{\ple{\scod{B}',\awfsRc F}}
						\ar[rd]_-D
&	\isosct{B}\ar[d]^{\ple{\scod{B},\sdom{B}}}\\
\sct{A}\times\sct{B}\ar[r]_-{F\times\sid{B}}&
\sct{B}\times\sct{B}}\]
with the notation of diagram (\ref{pulm}).
The structure map $S_D : \awfsMc D \to \to \awfsMc F$
is induced by those on $F$ and \ple{\scod{B},\sdom{B}} and maps
a pair $(A, x:B\isoa FA), (b_1,b_2):(B_1,B_2)\isoa(FA,B)$ to
the pair $S(A,b_1),b_1^{-1}xb_2:B_2\isoa B_1$.
A functor $T:\awfsMc F\to\sct{A}$ in the second component of the
carrier has to fit in the commutative diagram
\[\xymatrix{
\awfsMc F	\ar[d]_-D \ar[r]^-{T}
&	\sct{A}\ar[d]^-{F}
\\
\sct{B\times B}	\ar[r]^-{\pr2}	&	\sct{B}
}\]
and, since it has to be a homomorphism of algebras,
the following diagram must commute
\begin{equation}\label{dhom}
\vcenter{\xymatrix@C=5.5em{
\awfsMc D	\ar[r]^-{\awfsMc(\pr2,T)} \ar[d]_-{S_D}
&	\awfsMc F	\ar[d]^-{S}
\\
\awfsMc F	\ar[r]^-{T}	&	\sct{A}.
}}
\end{equation}
Moreover, the strictness condition imposes that the diagram
\begin{equation}\label{dstr}
\vcenter{\xymatrix@C=3em{
\sct{A}\ar[d]_-{\awfsLc F}\ar[rd]^-{\sid{A}}\\
\awfsMc F\ar[r]^(.49){T}&\sct{A}}}
\end{equation}
commutes.
Note also that there is an arrow $H:\awfsMc (\awfsRc F) \to \awfsMc D$
such that $S_D H = \mu_F$ and $\awfsMc(\pr2,T)H = \awfsMc(\sid{B},T)$.
Precomposing diagram (\ref{dhom}) with $H$
and using (\ref{dstr}) together with a
triangular identity for the monad, the commutative diagram
\[\xymatrix@C=2.5em{
\awfsMc F	\ar[rd]|-{\qquad\awfsMc(\sid{B},\awfsLc F)}
			\ar@/^18pt/[rrrd]^{\sid{\awfsMc F}}
			\ar@/_15pt/[rdd]_{\sid{\awfsMc F}}\\
&	\awfsMc(\awfsRc F)	\ar[rr]^-{\awfsMc(\sid{B},T)} \ar[d]^-{\mu_F}
&&	\awfsMc F	\ar[d]^-{S}
\\
&	\awfsMc F	\ar[rr]^-{T}	&&	\sct{A}.
}\]
shows that the only possible choice for $T$ is the structural 
functor $S:\awfsMc F\to\sct{A}$, and it is straightforward to see that
that choice makes diagrams (\ref{dhom}) and (\ref{dstr}) commute.
\end{proof}

\begin{cor}
The fibrations $\RcMap\ftr\Cat$ and $\RgMap\ftr\Gpd$ are not
elementary.
\end{cor}

\begin{proof}
The forgetful fibered functor $U':\RcAlg\ftr\RcMap$ takes loops to
loops. But, by \lmx\ref{plem}, a transporter for an algebra in \RcAlg
exists if and only if the algebra underlies an algebra for the monad.
The same argument applies to the fibration $\RgMap\ftr\Gpd$.
\end{proof}

\begin{rem}
Note that the argument in the proof of \prx\ref{prel} which shows that
arrows in $\BCFRcl{\pr{1,1}}{\pr1}(\pr{1,1},\srefl{B})$ are
\fibepic{\FTR{S}:\RcAlg\ftr\Cat}
 can be repeated to show that the same arrows are
\fibepic{\FTR{N}:\RcMap\ftr\Cat}.
Also, since the forgetful fibered functor
$U':\RcAlg\ftr\RcMap$ preserves finite products,
$\FTR{N}:\RcMap\to\Cat$ has the underlying arrows to ensure strictly
productive transporters.
\end{rem}

\section*{Acknowledgments}
We would like to thank Richard Garner, Martin Hyland, and John Power
for comments, suggestions and support in the early stages of this
work.

\bibliographystyle{alpha}
\bibliography{biblioEPR}

\end{document}